\documentclass[draft]{amsart}
\usepackage{amsmath,amsfonts,amssymb,amscd,verbatim,multicol}
\usepackage[arrow,matrix,cmtip]{xy}

\usepackage{fullpage}

\newcommand{\trdeg}{\operatorname{tr.deg}}

\newcommand{\bpr}{\begin{proof}}
\newcommand{\epr}{\end{proof}}

\newcommand{\spec}{\operatorname{Spec}}

\newcommand{\mc}{\mathcal}
\newcommand{\mf}{\mathfrak}
\newcommand{\cha}{\operatorname{char}}
\newcommand{\mb}{\mathbb}

\newcommand{\GK}{\operatorname{GKdim}}
\newcommand{\GKdim}{\operatorname{GKdim}}

\newcommand{\wt}{\widetilde}

\newcommand{\lGr}{\operatorname{-Gr}}

\newcommand{\rGr}{\operatorname{Gr-}\hskip -2pt}

\newcommand{\wh}{\widehat}

\newcommand{\beq}{\begin{equation}}
\newcommand{\eeq}{\end{equation}}

\newcommand{\Hom}{{\rm Hom}}

\newcommand{\End}{{\rm End}}

 \DeclareMathOperator{\Pic}{Pic}

\numberwithin{equation}{section}
 \theoremstyle{plain}
\newtheorem{theorem}[equation]{Theorem}

\newtheorem{lemma}[equation]{Lemma}

\newtheorem{proposition}[equation]{Proposition}

\theoremstyle{definition}
\newtheorem{question}[equation]{Question}
\newtheorem{definition}[equation]{Definition}

\newtheorem{remark}[equation]{Remark}

\newtheorem{hypothesis}[equation]{Hypothesis}
\newtheorem{standing-hypothesis}[equation]{Standing Hypothesis}
\newtheorem{example}[equation]{Example}

\title{$\mb{Z}$-graded simple rings}
\author{J. Bell and D. Rogalski}

\address{(Rogalski)
UCSD, Department of Mathematics, 9500 Gilman Dr. \# 0112, La Jolla,
CA 92093-0112, USA. } 
\email{drogalsk@math.ucsd.edu}

\address{(Bell) Department of Pure Mathematics, 
University of Waterloo, 
Waterloo, ON, 
CANADA N2L 3G1}
\email{jpbell@uwaterloo.ca}

\thanks{The first author was  partially   supported by NSERC grant 31-611456.}
\thanks{The second author was partially supported by NSF grants DMS-0900981 and DMS-1201572.}
\subjclass[2000]{16D30, 16P90, 16S38, 16W50}
\keywords{Noncommutative geometry,  graded ring, generalized Weyl algebra, simple ring}

\begin{document}

\begin{abstract}
The Weyl algebra over a field $k$ of characteristic $0$ is a simple ring of Gelfand-Kirillov dimension 2, which has a grading by the group of integers.    We classify all $\mb{Z}$-graded simple rings of GK-dimension 2 and show that they are graded Morita equivalent to generalized Weyl algebras as defined by Bavula.  More generally, we study $\mb{Z}$-graded simple rings $A$ of any dimension which have a graded quotient ring of the form $K[t, t^{-1}; \sigma]$ for 
a field $K$.    Under some further hypotheses, we classify all such $A$ in terms of a new construction of simple rings which we introduce in this paper.   In the important special case that $\GKdim A = \trdeg(K/k) + 1$, we show that 
$K$ and $\sigma$ must be of a very special form.   The new simple rings we define should warrant further 
study from the perspective of noncommutative geometry.
\end{abstract}

\maketitle

\section{Introduction}

Let $k$ be an algebraically closed field.  The Weyl algebra, $A = k \langle x, y \rangle/(yx-xy-1)$, is 
one of the most important and well-studied examples in noncommutative algebra. 
As is well-known, it has Gelfand-Kirillov (GK) dimension 2, and when $\cha k = 0$ it is a simple ring---in fact, in some sense 
it is the prototypical non-artinian simple ring.
The ring $A$ has a $\mb{Z}$-grading with $\deg x =1$, $\deg y = -1$ which has been exploited to interesting 
effect in some recent work.  In particular, the category $\rGr A$ of $\mb{Z}$-graded $A$-modules was shown 
by Paul Smith to be equivalent to the category of quasi-coherent sheaves on a certain stack of dimension one \cite{Sm}; thus, one 
may interpret this category as a noncommutative curve.    Smith's work was inspired by 
earlier work of Sue Sierra, who studied the properties of $\rGr A$ and showed that the class of $\mb{Z}$-graded rings 
with an equivalent graded module category is surprisingly varied, and even includes many non-simple examples \cite{Si}.
Much work in noncommutative algebraic geometry has concentrated on analogs of projective schemes, and in 
particular has focused on $\mb{N}$-graded algebras.  The results above suggest that it would be interesting  
to consider $\mb{Z}$-graded algebras, and the geometry of their graded module categories, more thoroughly.

In this paper, we focus on the problem of finding and studying other examples of $\mb{Z}$-graded rings that 
generalize the Weyl algebra in various ways.  In particular, this project originally began with following question:
what are the other $\mb{Z}$-graded simple domains of GK-dimension 2?  We give 
a complete answer to this question in the next theorem.  First, we review some definitions.  Recall that given a $\mb{Z}$-graded $k$-algebra $A$ which is an Ore domain, localizing at the set of nonzero homogeneous elements yields the \emph{graded quotient ring} $Q_{\rm gr}(A)$,  which has the form of a skew-Laurent ring $D[t, t^{-1}; \sigma]$ for some division ring $D$ with automorphism $\sigma$.
By a \emph{graded Morita equivalence} between two $\mb{Z}$-graded algebras $A$ and $B$, we mean an equivalence of 
the full module categories over these rings which also restricts to an equivalence of their subcategories of $\mb{Z}$-graded modules.

\begin{theorem}

\label{thm:GK2-intro} (Theorem~\ref{thm:GK2})
Let $k$ be an algebraically closed field and let $A=\bigoplus A_i$ be a $\mb{Z}$-graded finitely generated simple $k$-algebra which is a domain of GK-dimension $2$ with $A_i \neq 0$ for all $i \in \mb{Z}$.
\begin{enumerate}
\item The graded quotient ring of $A$ has the form $Q = Q_{\rm gr}(A) \cong K[t, t^{-1}; \sigma]$, where $K = k(u)$
is a rational function field in one variable.  Let $T = A_0$.  We can choose $u$ so that
either 
\begin{enumerate}
\item[(A)] $\sigma(u) = u+1$, $T = k[u]$, and $\cha k = 0$; or 
\item[(B)] $\sigma(u) =p u$ for some non-root of unity $p \in k^*$, and $T = k[u, u^{-1}]$.
\end{enumerate}  
\item $A$ is graded Morita equivalent to a generalized Weyl algebra $T(\sigma, f)$, for some $f \in T$ which does not have two distinct 
roots on any $\sigma$-orbit.   
\end{enumerate}
\end{theorem}
\noindent  
The generalized Weyl algebras are an interesting class of rings defined by Bavula in \cite{Bav} and studied extensively studied by him and others; one possible definition of these rings is given later in the introduction.

As it turns out, both parts of Theorem~\ref{thm:GK2-intro} have surprisingly strong generalizations to rings of higher dimension, 
and the theorem is obtained as a very special case of these more general results.  First we describe the generalization of part (1).  
The theorem implies in particular that the graded quotient ring $Q_{\rm gr}(A)$ is of the form $K[t, t^{-1}; \sigma]$ for a field $K$ of transcendence degree $1$ over $k$.  This is closely related to a well-known result of Artin and Stafford \cite[Theorem 0.1]{AS}, which proves that the graded quotient rings of finitely generated $\mb{N}$-graded domains of GK-dimension 2 must have this form.
In general, we say that a $\mb{Z}$-graded Ore domain $A$ is \emph{birationally commutative} if its graded quotient ring 
has the form $Q_{\rm gr}(A) \cong K[t, t^{-1}; \sigma]$ for a field $K$.   Since $\mb{Z}$-graded simple domains may be too large a class to analyze,  we attempt to generalize to higher dimension more narrowly by adding the hypothesis of birational commutativity. 
If $A$ is a birationally commutative $\mb{Z}$-graded domain with graded quotient ring 
$K[t, t^{-1}; \sigma]$, the GK-dimension of $A$ is at least as big as $\trdeg(K/k) +1$, but there are many examples where 
it is bigger (for instance, see Example~\ref{ex:weyl-loc} below).  It turns out that rings satisfying $\GKdim(A) = \trdeg(K/k) +1$
behave especially well, and so we add this as a hypothesis also.   Finally, it is 
convenient to assume that the center of $Q_{\rm gr}(A) = K[t, t^{-1}; \sigma]$ is as small as possible, that is, 
equal to $k$.  Morally, this should be thought of as a weaker hypothesis than assuming that $A$ is simple (for example, it is automatic if $A$ is  primitive, noetherian, and $k$ is uncountable).  With these restrictions, we can prove the 
following higher-dimensional analog of Theorem~\ref{thm:GK2-intro}(1).
\begin{theorem} (Theorem~\ref{thm:inT})
\label{low-GK-thm} 
Let $k$ be an algebraically closed field.
Suppose that $A$ is a finitely generated $\mb{Z}$-graded $k$-algebra which is an Ore domain with 
graded quotient ring $Q = Q_{\operatorname{gr}}(A) = K[t,  t^{-1}; \sigma]$, where $K$ is a field with 
$\trdeg K/k = d$ and $\GK(A) = d+ 1$.  Assume that $Q$ has center $k$.

Then $K = k(x_1, \dots, x_d)$ is a rational function field in indeterminates $x_i$ over $k$, and $A_0 \subseteq T$ where $\sigma(T) = T$ and one of the following two cases occurs:
\begin{enumerate}
\item[(A)]  $\sigma(x_1) = x_1 + 1$, $\sigma(x_i) = p_i x_i$ for all $i \geq 2$, for some $p_2, 
\dots, p_d$ which generate a free abelian subgroup of $k^{\times}$, $\cha k = 0$, and 
$T = k[x_1, x_2^{\pm 1}, \dots, x_m^{\pm 1}]$; or 
\item[(B)] $\sigma(x_i) = p_i x_i$ for all $i \geq 1$, for some $p_1, \dots, p_d$ which generate 
a free abelian subgroup of $k^{\times}$, and $T = k[x_1^{\pm 1}, x_2^{\pm 1}, \dots, x_m^{\pm 1}]$.
\end{enumerate}
Moreover, if $A$ is simple, then $A_0 = T$.
\end{theorem}

To describe our generalization of part (2) of Theorem~\ref{thm:GK2-intro}, we first need some notation.
Suppose that $R$ is a commutative noetherian $k$-algebra with an automorphism $\sigma: R \to R$.  
Let $X = \spec R$ and let $Z$ be a closed subset of $X$ such that $\sigma^i(Z) \cap Z = \emptyset$ for all $i \neq 0$.  We say that such a closed subset is \emph{$\sigma$-lonely}.  Let $H$ and $J$ be ideals of $R$ such that $\spec R/H$ and $\spec R/J$ are 
contained in $Z$ as sets.  We define a ring 
$B(Z, H, J)$ as follows:
\[
B(Z, H, J) = \bigoplus_{n \in \mb{Z}}  I_n t^n \subseteq R[t, t^{-1}; \sigma], 
\]
where $I_0 = R$, $I_n = J \sigma(J) \cdots \sigma^{n-1}(J)$ for $n \geq 1$, and 
$I_n = \sigma^{-1}(H) \sigma^{-2}(H) \cdots \sigma^{n}(H)$ for $n \leq -1$.
To give a specfic example, the generalized Weyl algebra $T(\sigma, f)$ appearing in Theorem~\ref{thm:GK2-intro} is isomorphic to $B(Z, H, T)$, where $R = T$ and $\sigma$ are as in that theorem, and $Z$ is the closed subset of $X = \spec T$ 
defined by the principal ideal $H = (\sigma^{-1}(f))$ (see the proof of Theorem~\ref{thm:GK2} below).
In Proposition~\ref{prop:Bsimp}, we will show that the ring $B(Z, H, J)$ is simple as long 
as $\sigma:X \to X$ is a \emph{wild} automorphism, that is, $X$ has no closed subsets $Y$ 
with $\sigma(Y) = Y$ other than $Y$ and $\emptyset$.  We also prove that $B(Z, H, J)$ is noetherian if 
the $\sigma$-orbit of the subset $Z$ is critically dense in $X$, but not strongly noetherian if $Z$ has codimension at least $2$ in $X$ 
(see Section~\ref{sec:somerings} for the definitions of these terms and some further discussion).

Our main classification result shows that in wide generality, birationally commutative simple $\mb{Z}$-graded 
domains must be very closely related to the rings $B(Z, H, J)$.
Given any $\mb{Z}$-graded subring $A$ of $K[t, t^{-1}; \sigma]$ with $A = R$, where $K$ is the field of fractions of $R$, 
and an invertible $R$-module $M \subseteq K$, we define a new algebra $A'  = \bigoplus_{n \in \mb{Z}} M_n A_n$ which we call the \emph{$\Pic(X)$-twist} of $A$ by $M$, where $M_n = M \sigma(M) \dots \sigma^{n-1}(M)$ and 
$M_{-n} = [\sigma^{-1}(M) \dots \sigma^{-n}(M)]^{-1}$ for $n \geq 0$.  The rings $A$ and $A'$ are not Morita equivalent in general, 
but they do have equivalent $\mb{Z}$-graded module categories.
\begin{theorem} (Theorem~\ref{thm:class})
\label{thm:general-intro}
Let $k$ be algebraically closed field.  
Let $A$ be a simple $\mb{Z}$-graded $k$-algebra which is finitely generated as an algebra and a birationally commutative Ore domain with  $Q_{\operatorname{gr}}(A) \cong K[t, t^{-1}; \sigma]$.  Assume that either $\cha k = 0$ or that $\trdeg K/k = 1$.  
Let $R = A_0$ and $X = \spec R$.  Suppose 
that $R$ is noetherian algebra such that (i) the integral closure of $R$ is a finite $R$-module; and (ii) the singular locus 
of $X = \spec R$ is a closed subset of $X$. 

Then $R$ is a regular ring with $\sigma(R) = R$, and $\sigma$ is a wild automorphism of $X$.  Moreover, some $\Pic(X)$-twist of $A$ is graded Morita equivalent to $B(Z, H, J)$, for some $\sigma$-lonely subset $Z \subseteq X$ and ideals $H, J$ such that $R/H$ and $R/J$ are supported along $Z$.  In particular, the categories of graded modules $\rGr A$ and $\rGr B(Z, H, J)$ are equivalent.
\end{theorem}

The assumptions on $\cha k$ and on $R = A_0$ in the theorem above may just be artifacts of our proof, as we have no examples showing that any of these assumptions is necessary.
In any case, (i) and (ii) are very weak assumptions which hold for all excellent rings, in particular for any finitely generated $k$-algebra $R$.  Unfortunately, the assumption that $A$ is finitely generated 
as an algebra does not imply that $R = A_0$ is finitely generated, and there are many examples satisfying the theorem for which 
$A_0$ is indeed infinitely generated as an algebra.  
On the other hand, in the important special case of Theorem~\ref{thm:general-intro} where 
we also assume the hypotheses of Theorem~\ref{low-GK-thm} (in particular that $\GKdim(A) = \trdeg(K/k) + 1$), then 
that theorem shows in particular that $A_0 = T$ is finitely generated, and so (i) and (ii) become automatic.

Birationally commutative connected $\mb{N}$-graded algebras have been studied extensively, particularly those of GK-dimension 3 where there is now a detailed classification; see \cite{RS} and \cite{Si}.   The analysis of such birationally commutative algebras in higher dimension seems to be a very difficult problem, so it is surprising that in the $\mb{Z}$-graded simple setting we 
are able to prove a dimension-independent structure theorem such as Theorem~\ref{thm:general-intro}.  
Birationally commutative $\mb{N}$-graded algebras have provided many examples of rings which are noetherian but not strongly noetherian, and the rings $B(Z, H, J)$ show that there are also such examples which are simple and $\mb{Z}$-graded.

To close, we discuss a few further questions.  
First, we have not devoted a lot of study to the representation theoretic or geometric properties 
of the categories of graded modules over the simple rings we construct in this paper.
Given a ring 
of the form $B(Z, H, J)$, is its category of $\mb{Z}$-graded modules equivalent to the category of quasi-coherent sheaves on some stack, as in Smith's work on the Weyl algebra \cite{Sm}?  
Next, in order to fully understand what simple rings  $B(Z, H, J)$ are possible, one would like to understand what are the 
wild automorphisms $\sigma$ of regular noetherian schemes  $X = \spec R$, and which 
closed  subsets $Z$ are $\sigma$-lonely.   The question about wild automorphisms is 
a close counterpart to a similar question about wild automorphisms of projective varieties studied in \cite{RRZ}; see Remark~\ref{rem:wild} below 
for more discussion.   We do study here the question of which subsets are $\sigma$-lonely in the case 
that $X = \spec T$ and $\sigma$ are of the form occurring in Theorem~\ref{low-GK-thm}.   In this 
case we can fully characterize $\sigma$-lonely subsets $Z$ of codimension $1$ in $X$ (see Theorem~\ref{thm:lonely} below), but 
the characterization for $Z$ of arbitrary codimension is open.
Finally, one of our hopes for this project was to produce some interesting new examples of simple algebras which might serve as a testing 
ground for conjectures about simple rings.  One famous such open question is the following:  can every right ideal of a simple noetherian ring be generated by at most 2 elements \cite[Appendix, Question 19]{GW}?   It would be interesting to study this question for the simple rings of the form $B(Z, H, J)$.

\section*{Acknowledgements} We thank Sue Sierra, Lance Small, Paul Smith, and Toby Stafford for helpful conversations.

\section{Some simple $\mb{Z}$-graded algebras}
\label{sec:somerings} 
\emph{Throughout this paper, $k$ will be an algebraically closed field.}  In some later results it will be convenient to assume further that $k$ is uncountable or of characteristic $0$.   All rings in this paper will be algebras over the field $k$, and all schemes will be $k$-schemes, 
though we will not always emphasize this explicitly.

In this section, we construct some interesting examples of simple algebras, and study their properties.   Some special 
cases of our construction include well-known examples such as generalized Weyl algebras and rings Morita equivalent to them, 
but to our knowledge our class of examples has not been considered previously as a whole.

To begin the section, we define the useful notion of cycle on an orbit of a closed subset.
\begin{definition}
\label{def:cycle}
Let $X$ be a scheme with automorphism $\sigma: X \to X$.  For any closed subset $Z$ of $X$, we define
$Z_i = \sigma^{-i}(Z)$ for each $i \in \mb{Z}$.   
A \emph{cycle} supported on the orbit of $Z$ will be an element of the free abelian group on the $Z_i$.   
Essentially we treat the $Z_i$ as formal symbols, which should be thought of as distinct regardless of whether the corresponding closed subsets are.  However, in our intended applications we will usually have $Z_i \cap Z_j = \emptyset$ for $i \neq j$ anyway.

We call a cycle \emph{effective} if all of its coefficients are nonnegative. 
If $D$ and $E$ are cycles then we write $D \geq E$ if $D - E$ is effective.  We write $\max(D, E)$
for the smallest cycle $F$ such that $F-D$ and $F-E$ are effective, and define $\min(D, E)$ similarly.
Given a cycle $D = \sum_i a_i Z_i$, we write $\sigma^j(D)$ for the cycle 
$\sum_i a_i \sigma^j(Z_i) =  \sum_i a_i Z_{i-j} = \sum_i a_{j+i} Z_i$.
\end{definition}
\noindent  The symbols $Z_i$ in a cycle $\sum a_i Z_i$ are primarily placeholders and the combinatorics of the integers $a_i$ will be our main concern.  Typically $a_i$ will measure the multiplicity of vanishing of a function or ideal along the closed subset $\sigma^{-i}(Z)$.
We will not apply intersection theory to cycles.

 \begin{definition}
\label{B-def}
\label{def:Gn}
Let $X$ be a scheme, let $\sigma: X \to X$ be an automorphism and let $Z \subseteq X$ be a closed subset.
Let $G$ be any cycle $\sum a_i Z_i$.  We define for each $n \in \mb{Z}$ 
a cycle $G_{n}$, as follows.  Set $G_{0} = 0$.  For $n \geq 1$,  let $G_{n} =  G + \sigma^{-1}(G) + \dots + \sigma^{-n+1}(G)$, and let $G_{-n} = -\sigma(G) - \sigma^2(G) - \dots - \sigma^n(G)$.  
\end{definition}
We have the following trivial properties of this definition, whose proofs we leave to the reader.
\begin{lemma}
\label{lem:easyGfacts} Fix a cycle $G$ on the $\sigma$-orbit of $Z$ and define $G_n$ as in Definition~\ref{def:Gn}.
\begin{enumerate}
\item  $G_{n} = - \sigma^{-n}(G_{-n})$ for 
any $n \in \mb{Z}$.

\item $G_{m} + \sigma^{-m}(G_{n}) = G_{m+n}$ for all $m,n \in \mb{Z}$. \hfill $\Box$
\end{enumerate}
\end{lemma}

The special cycles in the following definition will play a crucial role below.  
\begin{definition}
A cycle of the form $G  = \sum_{i=m}^n a_i Z_i$ with $m \leq n$ is \emph{pleasantly alternating}
if $a_m = 1 = a_n$ and the nonzero $a_i$ with $m \leq i \leq n$
alternate strictly between $1$ and $-1$.    We say that $G$ is a \emph{trivial} pleasantly alternating sequence 
if $m = n$ and so $G = Z_m$.
\end{definition}
\noindent
 For example,
$Z_{-3} - Z_{-1} + Z_0 - Z_5 + Z_6$ is pleasantly alternating.   

We work out some of the basic combinatorial properties of the cycles $G_{n}$, where $G$ is pleasantly alternating, in the next two results.
\begin{lemma}
\label{G-lem}
Let $X$ be a scheme with automorphism $\sigma$ and let $Z \subseteq X$ be a closed subset.
Let $G = \sum_{i = r}^s g_i Z_i$  be a pleasantly alternating cycle on the orbit of $Z$ with 
with $g_r = g_s = 1$, and let $N = s-r$.
Write $G_{n} = \sum_i g_{n,i} Z_i$.  

\begin{enumerate}
\item $G_{n}$ has all of its coefficients in $\{-1, 0, 1 \}$, for $n \in \mb{Z}$.  Moreover, 
$G_{n}$ has all of its coefficients in $\{0, 1 \}$ for $n \geq N$ and  in $\{ 0 , -1 \}$ for $n \leq -N$.

\item For any $i \in \mb{Z}$, $g_{n,i}$ is constant for all $n \gg 0$, 
say $g_{n, i} = a$, for all $n \gg 0$.  Similarly, $g_{n,i} = b$ for all $n \ll 0$, 
some $b$.  Either $a=0$ and $b = -1$, or $a=1$ and $b = 0$.

\item For $n \geq N$, $\min(G_{n},  -G_{-n}) = \min(G_{n}, \sigma^n(G_{n})) = 0$.  
\end{enumerate}
\end{lemma}
\begin{proof}
(1)   
By definition, for $n \geq 0$ 
we have 
\[
G_{n} = \sum_{j = 0}^{n-1} \sigma^{-j}(G) = 
\sum_{j =0}^{n-1} \sum_i g_i Z_{i+j}
= \sum_k \sum_{j = 0}^{n-1} g_{k-j} Z_k,
\]
from which we see that $G_{n} = \sum_i g_{n,i} Z_i$ where 
$g_{n,i} = \sum_{j = 0}^{n-1} g_{i-j} = \sum_{j = i-n + 1}^i g_j$. 
Each such number is a sum of consecutive coefficients in $G$ and, 
since $G$ is pleasantly alternating, is a number in $\{-1, 0, 1 \}$.  
Moreover, if we put $d_i = \sum_{\{j \in \mb{Z} | j \leq i\}} g_j$ for each $i \in \mb{Z}$,
then clearly $g_{n,i} = d_i$ for all $n \gg 0$.  Also, 
each $d_i \in \{0, 1 \}$ since $G$ is pleasantly alternating.
If $n \geq N$, then $g_{n,i}$ is a sum of at least $N$ consecutive coefficients of $G$
and so must lie in $\{0, 1 \}$.  Thus  $G_{n}$ is effective for $n \geq N$.

If instead $n < 0$, since $G_{n} = - \sigma^{-n}(G_{-n})$ the result 
follows.

(2) the proof of (1) showed that $g_{n,i}$ is constant equal to $d_i$ for $n \gg 0$.
We claim that for any $i$, $g_{n, i+n}$ is also constant for $n \gg 0$.  
Since $g_{n, i+n} = \sum_{j= i+1}^{i+n} g_j$, we get for $n \gg 0$ 
that 
$g_{n, i+n} = e_i = \sum_{j \geq i+1} g_j$.  
Note that $d_i + e_i = \sum_{j \in \mb{Z}} g_j = 1$.

Now it is easy to check that the equation $G_{n} = - \sigma^{-n}(G_{-n})$ implies 
that $g_{n,i} = - g_{-n, i-n} =  - e_i$ for $n \ll 0$.  
So in the notation of the lemma, $a = d_i$ and $b = -e_i$.
Since $d_i + e_i = 1$, 
either $a = d_i = 0$ and $b = -e_i = -1$, or else $a = d_i = 1$ and $b =  -e_i = 0$.  

(3)  This follows from a similar calculation as in part (2).  Namely, for $n \geq N$ we need $g_{n,i}$ and $- g_{-n,i}$ to 
not both be 1.  We calculated that $g_{n,i} = \sum_{j = i - n + 1}^i g_j$ and 
$-g_{-n,i} = g_{n, i+n} = \sum_{j = i + 1}^{i+n} g_j$.  By the definition of pleasantly alternating 
sequence, these numbers cannot both be $1$.
\end{proof}

There is a kind of converse to Lemma~\ref{G-lem}(3) which we give next.
\begin{lemma}
\label{lem:converse} Let $Z \subseteq X$ be a closed subset of a scheme $X$ with automorphism $\sigma$.
Let $G = \sum_{i = r}^s g_i Z_i$ be a cycle supported on the orbit of $Z$ and let $N = s-r$.  

Define $G_n$ as in Definition~\ref{def:Gn} and suppose that $G_{n}$ is effective for all $n \gg 0$.  Then either there is $r \leq i \leq s-1$ such that $Z_i \leq \min(G_{n}, \sigma^n(G_{n}))$ 
for all $n \geq N$, or else $\min(G_{n}, \sigma^n(G_{n})) = 0$ for all $n \geq N$.  In the latter case, 
$G$ is a nonnegative multiple of a pleasantly alternating cycle.
\end{lemma}
\begin{proof}
By shifting indices we may assume that $G$ is of the form $G = \sum_{i = 0}^N g_i Z_i$.
Write $G_{n} = \sum_i g_{n,i} Z_i$ for each $n$, and let 
$d_i = \sum_{j \leq i} g_j$ and $e_i = \sum_{j > i} g_j$ as in the proof of Lemma~\ref{G-lem}.  Let $d = \sum_{j \in \mb{Z}} g_j$.     The same calculation as in Lemma~\ref{G-lem} shows that for $n \geq N$
we have 
\begin{equation}
\label{eq:comb}
G_{n} = d_0 Z_0 + d_1 Z_1 + \dots + d_{N-1} Z_{N-1} +
d Z_N + d Z_{N+1} + \dots + d Z_{n-1} + e_0 Z_n + e_1 Z_{n+1}
+ \dots + e_{N-1} Z_{n + N -1}.
\end{equation}
The hypothesis that $G_{n}$ is effective for $n \gg 0$ now forces $d_i \geq 0$ and $e_i \geq 0$ for all $i$.
Taking $i$ in the range $0 \leq i \leq N-1$, we see from \eqref{eq:comb} that 
$\min(G_{n}, \sigma^n(G_{n})) \geq Z_i$ for all $n \geq N$, if $e_i > 0$ and $d_i > 0$.

Otherwise, since $e_i + d_i = d$  for each $i$ we must have either $d_i = d$ and $e_i = 0$, or else $d_i = 0$ and $e_i = d$.  
Then from \eqref{eq:comb} we get $\min(G_{n}, \sigma^n(G_{n}))  = 0$ for $n \geq N$.
Moreover, since $d_{-1} = 0$, $d_N = d$, and $d_{i+1} = d_i + g_{i+1}$ for each $i$, we see that $g_i \in \{-d, 0, d \}$ for $0 \leq i \leq N-1$.  Finally, the only way we can obtain $d_i \in \{0, d \}$ for all $i$ is for $G$ to be $d$ times a pleasantly alternating cycle.
\end{proof}

Before we define the simple rings of main interest in this section, we need one more definition.
\begin{definition}
Let $X$ be a scheme, and let $\sigma: X \to X$ be a automorphism.   A closed subset $Z$ of $X$ is \emph{$\sigma$-{lonely}} (or just \emph{lonely} if $\sigma$ is understood) if 
$Z \cap \sigma^i(Z) = \emptyset$ for all $0 \neq i \in \mb{Z}$. 
\end{definition}

\begin{definition}
\label{def:B}
Let $X = \spec R$, where $R$ is a noetherian commutative domain which is a $k$-algebra with $k$-automorphism $\sigma: R \to R$.  
Let $\sigma: X \to X$ also denote the corresponding automorphism of $X$.  Let $Z$ be a $\sigma$-lonely closed subset of $X$ and let $Z_i = \sigma^{-i}(Z)$ for all $i \in \mb{Z}$.  Let $G$ be a pleasantly alternating cycle on the orbit of $Z$, and let $G_n$ be defined for each $n \in \mb{Z}$
as in Definition~\ref{def:Gn}.  Let $H$ and $J$ be ideals of $R$ defining closed subsets contained in $Z$.  Note that $R/\sigma^i(H)$ and $R/\sigma^i(J)$ are supported along $\sigma^{-i}(Z) = Z_i$.

Given a cycle $D = \sum a_i Z_i$ on the orbit of $Z$ with coefficients $a_i \in \{0, 1\}$, we put $H[D] =  \prod_{\{i \in \mb{Z} | a_i = 1\}} \sigma^i(H)$, and similarly for the ideal $J$.  Note that $\sigma^i(H[D]) = H[\sigma^{-i}(D)]$.
For an arbitrary cycle $D$ we put $D^+ = \max(D, 0)$.

Now we define 
\[
B(G, H, J) = \bigoplus I_n t^n \subseteq R[t, t^{-1}; \sigma],
\]
where $I_{n} = H[(-G_{n})^{+}] J[G_{n}^{+}]$ for each $n$.  
We allow the special case $H = R$ or $J = R$, and in fact we immediately have the useful 
decomposition
\[
B(G, H, J) = B(G, R, J) \cap B(G, H, R).
\]
\end{definition}

We will verify that $B(G, H, J)$ is a ring in the next result, which also gives some basic properties of these rings.  Deeper properties will be proved later in the section,  after we have studied Morita equivalence for these algebras.
\begin{lemma}
\label{lem:B-props1}
Assume the setup and notation of Definition~\ref{def:B}.  Write $G = \sum_{i=r}^s a_i Z_i$ for some $r \leq s$ with $a_r = a_s = 1$, 
and let $N = s-r$.
\begin{enumerate}
\item $B = B(G, H, J)$ is a subring of $R[t, t^{-1}; \sigma]$.

\item $B_m B_n = B_{m+n}$ for all $m,n \geq N$ and $m,n \leq -N$.   In particular, $B$ is generated as a $k$-algebra 
by $\bigoplus_{n = -m}^m B_n$, where $m = 2N-1$ if $N > 0$ and $m = 1$ if $N = 0$.
If $R$ is a finitely generated $k$-algebra, then $B$ is a finitely generated $k$-algebra.

\item $B_n B_{-n} + B_{-n} B_n$ is the unit ideal of $R = B_0$, for all $n \geq N$.

\item 
Any right (or left) ideal $I$ of $B$ containing $B_{\leq -m} \oplus B_{\geq m}$ for some $m \geq 0$ is the unit ideal of $B$.
As a consequence, any $\mb{Z}$-graded right module $M$ is generated as a module by $M_{\leq -n} \oplus M_{\geq n}$ for 
any $n \geq 0$.

\end{enumerate}
\end{lemma}
\begin{proof}
(1) It is enough to prove that $B(G, R, J)$ and $B(G, H, R)$ are subrings, 
and by symmetry we only need to do this only for $B = B(G, R, J)$.
We need $B_m B_n \subseteq B_{n+m}$,  or equivalently  
$J[G_{m}^+]t^m J[G_{n}^+]t^n \subseteq J[G_{m+n}^+]t^{m+n}$.  
This in turn is equivalent to $J[G_{m}^+] \sigma^m(J[G_{n}^+]) = J[G_{m}^+] J[\sigma^{-m}(G_{n}^+)] \subseteq J[G_{m+n}^+]$. 
This inclusion follows from the equation $G_{m} + \sigma^{-m}(G_{n}) = G_{m+n}$, together 
with the following observation: since all of the cycles $G_i$ have coefficients in $\{-1, 0, 1 \}$ by Lemma~\ref{G-lem}(1), 
if $G_{m+n}$ has a coefficient of $1$ along $Z_i$, then either 
$G_{m}$ or $\sigma^{-m}(G_{n})$ must have a coefficient of $1$ along $Z_i$ as well.

(2)  By Lemma~\ref{G-lem}(1), $G_{n}$ is effective for $n \geq N$, so $B_n = J[G_{n}]t^n$ for $n \geq N$.  
Then since $G_{m} + \sigma^{-m}(G_{n}) = G_{n+m}$ it follows that $B_m B_n = B_{n+m}$ for $n, m \geq N$.
The claims for negative degrees follow symmetrically.  
It is now easy to see that $B$ is generated as an algebra 
by $\bigoplus_{n =-m}^m B_n$ for the given $m$.

Since $R$ is noetherian, and each $B_i = V_i t^i$ with $V_i \subseteq R$ by definition, 
each $B_i$ is a finitely generated $R$-module.  So if $R$ is a finitely generated algebra, then clearly $B$ will be generated 
as an algebra by the union of a finite generating set for $R$ and a finite $R$-module generating set of $\bigoplus_{n = -m}^m B_n$.

(3) For  $n \geq N$ we have  $B_n = J[G_{n}]t^n$ and $B_{-n} = H[-G_{-n}]t^{-n}$.  Then 
\begin{gather*}
B_{-n} B_n + B_n B_{-n} = 
H[-G_{-n}] \sigma^{-n}(J[G_{n}]) + J[G_{n}] \sigma^n(H[-G_{-n}]) \\
= 
H[-G_{-n}]J[\sigma^n(G_{n})] + J[G_{n}] H[-\sigma^{-n}(G_{-n})] = HJ[\sigma^n(G_{n})] + HJ[G_{n}],
\end{gather*}
since $- \sigma^{-n}(G_{-n}) = G_n$.  Now because $\min(G_n, \sigma^n(G_n)) = 0$ by Lemma~\ref{G-lem} and $Z$ is $\sigma$-lonely, $HJ[\sigma^n(G_n)]$ and $HJ[G_n]$ are comaximal.

(4) The first part is immediate from (3).
For the statement about the module $M$, note that for fixed $n \geq 0$, then any $x \in M$ satisfies 
$x (B_{\leq -m} \oplus B_{\geq m}) \subseteq M_{\leq -n} \oplus M_{\geq n}$ for $m \gg 0$.  Thus any element 
of $M/(M_{\leq -n} \oplus M_{\geq n})B$ has a right annihilator which contains  $(B_{\leq -m} \oplus B_{\geq m})$
for some $m$, and hence is the unit ideal.  
\end{proof}

The next goal is to show that some of the rings $B(G, H, J)$ are Morita equivalent.  This also helps to explain the somewhat
complicated definition of these rings, as we will see explicitly how $B(G, H, J)$ arises as an endomorphism ring 
of a progenerator over a ring $B(Z, H, J)$.  In fact we will prove 
that $B(G, H, J)$ and $B(Z, H, J)$ are \emph{graded Morita equivalent}; that is, there is a 
Morita equivalence which is implemented by graded bimodules, and thus gives an equivalence of module categories that restricts 
to an equivalence of the graded module categories.

We need a few preliminary definitions and results.
For any Ore domain $A$ with Goldie quotient ring $Q(A) = D$, given right $A$-submodules $M,N \subseteq D$ we 
may identify $\Hom_A(M,N)$ with $\{ x \in D | x M \subseteq N \}$.  A similar comment holds for left submodules.  
If $A$ is a $\mb{Z}$-graded Ore domain, then the set of nonzero homogeneous elements in $A$ is easily 
checked to also satisfy the Ore condition.  The localization of $A$ at the set of nonzero homogeneous elements 
is the \emph{graded ring of fractions} $Q = Q_{{\rm gr}}(A)$.  Assuming that $A_1 \neq 0$, the graded ring of fractions has the form of a skew Laurent polynomial ring $Q_{\rm gr}(A) = D[t, t^{-1}; \sigma]$, for some division ring $D$ and $t$ of degree 1.  If  $M,N \subseteq Q$ are $\mb{Z}$-graded $A$-modules, we may identify $\Hom_A(M,N)$ with $\{ x \in Q | x M \subseteq N \}$.
\begin{lemma}
\label{lem:hom}
Assume the setup and notation from Definition~\ref{def:B}. 
Let $B, C$ be cycles on the orbit of $Z$ with 
coefficients in $\{0, 1 \}$.   Then 
\[
\Hom_R(H[B], H[C]) \cap R = H[\max(C-B, 0)].
\]
\end{lemma}
\begin{proof}
We have $\Hom_R(H[B], H[C]) = \{ x \in K | xH[B] \subseteq H[C] \}$, where $K$ is the field of fractions of $R$.   Then since $Z$ is $\sigma$-lonely, writing $B = \sum b_i Z_i$ and $C = \sum c_i Z_i$ we have\[
\Hom_R(H[B], H[C]) \cap R = \bigcap_i \big[ \{ x \in K | x (\sigma^i(H))^{b_i} \subseteq (\sigma^i(H))^{c_i} \} \cap R \big],
\]
where an ideal to the power $0$ is interpreted to be $R$.  The $i$th term in this intersection is clearly equal to $\sigma^i(H)$ if $b_i = 0$ and $c_i = 1$, and is equal to $R$ otherwise.  
\end{proof}

\begin{lemma}
\label{lem:anti}
Assume the setup and notation from Definition~\ref{def:B}.  
Then there is an anti-automorphism $\psi: R[t, t^{-1}; \sigma] \to R[t, t^{-1}; \sigma]$ given 
by the formula $xt^n \mapsto \sigma^{-n}(x) t^{-n}$.  
The map $\psi$ restricts to an anti-isomorphism $B(G, H, R) \to B(G, R, H)$.  
\end{lemma}
\begin{proof}
That $\psi$ is an anti-automorphism of $R[t, t^{-1}; \sigma]$ is routine.

For the second statement, recall that $B(G, H, R)_n = H[(-G_n)^+]t^n$ and $B(G, R, H)_n = H[G_n^+]t^n$.  Then  
\[
\psi(B(G, H, R)_n) = \sigma^{-n}(H[(-G_n)^+])t^{-n} = H[(-\sigma^n(G_n))^+)]t^{-n} = H[(G_{-n})^+]t^{-n} = B(G, R, H)_{-n},
\]
which implies the result.
\end{proof}

We now prove our main result concerning Morita equivalence.
\begin{proposition}
\label{prop:Morita}
Assume the setup and notation from Definition~\ref{def:B}.  
Then for any $i \in \mb{Z}$ there is a graded Morita equivalence between $B(G, H, J)$ and $B(\wt{G}, H, J)$, where $\wt{G} = Z_i$ is a trivial pleasantly alternating cycle.
\end{proposition}
\begin{proof}
We concentrate first on the case where $J = R$.
Let $A = B(\wt{G}, H, R)$, where $\wt{G} = Z = Z_0$.   Let $E_n = \max(-\wt{G}_n, 0)$, where $\wt{G}_n$ is defined 
in terms of $\wt{G}$ as in Definition~\ref{def:Gn}.   Then $A = \bigoplus_{n \in \mb{Z}} H[E_n] t^n$, where $E_n = 0$ for $n \geq 0$ and $E_{n} = Z_{n} + \dots + Z_{-1}$ for $n \leq -1$.

For each $j \geq 0$ we define a right ideal $L^{(j)}$ of $A$.  Define  $F^{(j)}_n = E_n + Z_j$, for $n \leq j$, and $F^{(j)}_n = E_n = 0$ for $n > j$.  Then let $L^{(j)} = \bigoplus_{n \in \mb{Z}} H[F^{(j)}_n] t^n$.   To see that $L^{(j)}$ is a right ideal, 
note that $A$ is generated in degrees $-1, 0,$ and $1$, by Lemma~\ref{lem:B-props1}(2).
Obviously each graded piece of $L^{(j)}$ is an $A_0 = R$-module.  
  The condition $L^{(j)}_n A_1 \subseteq L^{(j)}_{n+1}$ translates to $H[F^{(j)}_n] R \subseteq H[F^{(j)}_{n+1}]$, 
which is clear since $F^{(j)}_n \geq F^{(j)}_{n+1}$ for all $n$.    The condition $L^{(j)}_{n+1} A_{-1} \subseteq L^{(j)}_n$ 
requires $H[F^{(j)}_{n+1}] \sigma^{n+1}(\sigma^{-1}(H)) \subseteq H[F^{(j)}_{n}]$, 
which is true as long as the relation on divisors $F^{(j)}_{n+1} + Z_n \geq F^{(j)}_n$ holds.  For $n < 0$ this relation holds 
since $E_{n+1} + Z_j + Z_n = E_n + Z_j$.  The only other nontrivial case is $n = j$, for which the relation holds since $0 + Z_j \geq Z_j$.  Now let $S \subseteq \mb{N}$ be any finite set of nonnegative integers and define $L = \bigcap_{j \in S} L^{(j)}$.  Let $F_n = E_n + \sum_{\{j \in S | n \leq j \}} Z_j$ for all $n$, so we also have  $L = \bigoplus_n H[F_n] t^n$.

We compute $B = \operatorname{End}_A(L_A) = \{ x \in K[t, t^{-1}; \sigma] \, \big| \, x L \subseteq L \}$.    
By Lemma~\ref{lem:B-props1}(4), fixing any $n \geq 0$, then $L$ is generated by $L_{\leq -n} \oplus L_{\geq n}$.
Thus $x \in B$ if and only if $xL_{\leq n} \subseteq L$ for all $n \ll 0$ and $x L_{\geq n} \subseteq L$ for all $n \gg 0$.   
Write $D = \sum_{j \in S} Z_j$.
If $x = ft^m \in B_m$, then the equation $xL_n \subseteq L$ translates to $f\sigma^m(R) \subseteq R$ as long as 
$n, m+n > \max(S)$.  This is equivalent to $f \in R$.  On the other hand, the equation $x L_n \subseteq L$ 
translates to $f\sigma^m(H[E_n + D]) \subseteq H[E_{n+m} + D]$ as long as $n, n+m < 0$.  This is equivalent 
to $f \in \Hom_R ( H[\sigma^{-m}(E_n) + \sigma^{-m}(D)], H[E_{n+m} + D])$.  
By Lemma~\ref{lem:hom}, for $n \ll 0$ we have both $xL_{-n} \subseteq L$ and $xL_n \subseteq L$
if and only if $f \in H[C_m]$ where 
\[
C_m = \max(E_{n+m} + D - \sigma^{-m}(E_n) - \sigma^{-m}(D), 0) = \max(-\wt{G}_m + D - \sigma^{-m}(D), 0),
\]
where we use that for any $n \ll 0$, 
$E_{n+m} - \sigma^{-m}(E_n) = -\wt{G}_{n+m} + \sigma^{-m}(\wt{G}_n) = -\wt{G}_m$.  
Since $C_m$ is independent of $n \ll 0$, we see that $B_m = H[C_m]t^m$.
Now defining $G = Z + \sigma^{-1}(D) - D = \wt{G} + \sigma^{-1}(D) - D$, we have 
$G_n = \wt{G}_n + \sigma^{-n}(D) - D$ for all $n \in \mb{Z}$, where $G_n$ is defined in terms of $G$ as in Definition~\ref{def:Gn}.  In other words, $C_n = \max(-G_n, 0) = (-G_n)^+$, and $B = B(G, H, R)$ by definition.

Next, consider the left $A$-module $M = \Hom_A(L, A)$.  A very similar calculation to the one in the previous paragraph 
shows that $M_m = H[Y_m]t^m$ for 
$Y_m = \max(-\wt{G}_m - \sigma^{-m}(D), 0)$.
Note that $L_n = M_n = A_n = Rt^n$ for all $n \gg 0$.  It follows that the ideal $ML$ of $A$ contains $A_n = Rt^n$ for 
all $n \gg 0$, and so $ML$ is the unit ideal of $A$ by Lemma~\ref{lem:B-props1}(3).  Similarly, $LM$ is the unit ideal of $B$.
This shows that the bimodules 
$_B L_A$ and $_A M_B$ give a Morita equivalence between $A = B(\wt{G}, H, R)$ and $B = (G, H, R)$ \cite[Corollary 3.5.4]{MR}.  Since these bimodules are graded, it is a graded Morita equivalence.  

The above results hold just as well, of course, using the ideal $J$ in place of $H$.  Let $\wh{A} = B(\wt{G}, J, R)$, $\wh{L} = \bigoplus_{n \in \mb{Z}} J[F_n]t^n$, and $\wh{M} = \Hom_{\wh{A}}(\wh{L}, \wh{A})$.
By Lemma~\ref{lem:anti}, the anti-isomorphism $\psi$ of $R[t, t^{-1}; \sigma]$ restricts to an anti-isomorphism $\wh{A} = B(\wt{G}, J, R) \to B(\wt{G}, R, J)$.  Thus $L' = \psi(\wh{L})$ is a left ideal, and $M' = \psi(\wh{M})$ is a right module, over the ring $A' = B(\wt{G}, R, J)$.
Moreover,  $B' = \End({}_{A'} L') = \psi(\End(L_A)) = \psi(B(G, J, R)) = B(G, R, J)$, using Lemma~\ref{lem:anti} again.
We have $M'_{n} = L'_{n} = A'_n = Rt^{n}$ for all $n \ll 0$.   The bimodules  $_{A'} L'_{B'}$ and $_{B'} M'_{A'}$ must satisfy $L'M' = A'$ and $M'L' = B'$ and thus give a graded Morita equivalence between $A'$ and $B'$.  Note that $\End_{A'}(M'_{A'}) = B'$ as well.

To prove the result in general, we consider $A'' = B(\wt{G}, H, J) = B(\wt{G}, H, R) \cap B(\wt{G}, R, J) = A \cap A'$.
Since $L$ is a right ideal of $A$ and $M'$ is a right module over $A'$, $N = L \cap M'$ is a right module over $A''$.  Let $P = \Hom_{A''}(N, A'')$ and put $B'' = \End_{A''}(N)$.   
We claim now that 
\begin{equation}
\label{eq:ML}
M_{-n} M'_n + L'_n L_{-n} = R,\ \ \ \text{and}\ \ \   M'_n M_{-n} + L_{-n} L'_n = R, \ \ \text{for all}\ n \gg 0. 
\end{equation}
We show only the first equation; the proof of the second is similar.  The calculations above give the following explicit formulas for $n \gg 0$:
$L_{-n} = H[Z_{-n} + \dots + Z_{-1} + D]t^{-n}$, and $M_{-n} = H[Z_{-n} + \dots + Z_{-1} - \sigma^n(D)]t^{-n}$.
Using the formula for the anti-isomorphism $\psi$, we immediately get the following corresponding formulas for $n \gg 0$: 
$L'_n = J[Z_0 + \dots + Z_{n-1} + \sigma^{-n}(D)]t^n$, and $M'_n = J[Z_0 + \dots + Z_{n-1} - D]t^n$.
Now for $n \gg 0$, 
\[
M_{-n} M'_n= H[Z_{-n} + \dots + Z_{-1} - \sigma^n(D)]\sigma^{-n}(J[Z_0 + \dots + Z_{n-1} - D)]) = 
HJ[Z_{-n} + \dots + Z_{-1} - \sigma^n(D)],
\]
while 
\[
L'_n L_{-n} = J[Z_0 + \dots + Z_{n-1} + \sigma^{-n}(D)] \sigma^n(H[Z_{-n} + \dots + Z_{-1} + D)]) = 
HJ[Z_0 + \dots + Z_{n-1} + \sigma^{-n}(D)].  
\]
For $n \gg 0$ we have $\min(Z_{-n} + \dots + Z_{-1} - \sigma^n(D), Z_0 + \dots + Z_{n-1} + \sigma^{-n}(D)) = 0$, so 
the ideals $M_{-n} M'_n$ and $L'_n L_{-n}$ are comaximal in $R$ as claimed. 

Now note that for $n \gg 0$ we have $(L' \cap M)_{-n} = M_{-n}$, $(L' \cap M)_n = L'_n$, $N_{-n} = L_{-n}$, and $N_{n} = M'_n$.
Since $L' \cap M \subseteq P$, it 
follows easily from \eqref{eq:ML} that $PN = A''$ and $NP = B''$.   
Thus the rings $A'' = B(\wt{G}, H, J)$ and $B'' = \End_{A''}(N)$ are graded Morita equivalent.  
To calculate $B''$, note first that setting $U = R[t, t^{-1}; \sigma]$, then $NU$ is the unit ideal of $U$, since  $1 \in NP \subseteq NU$.
Thus $\End_{A''}(N) \subseteq \End_{U}(NU) = U$.   This implies that if $x \in \End_{A''}(N)$, 
then the conditions $x L_n \subseteq L$ and $x M'_{-n} \subseteq M'$ are trivially satisfied 
for $n \gg 0$.   Since $N_{\leq -n} = L_{\leq -n}$  and $N_{\geq n} = M'_{\geq n}$ for $n \gg 0$, 
we also have $x L_{\leq -n} \subseteq L$ and $x M'_{\geq n} \subseteq M'$ for $n \gg 0$.
Thus $x \in \End_A(L)$ and $x \in \End_{A'}(M')$ by Lemma~\ref{lem:B-props1}(4).
It follows that $\End_{A''}(N) \subseteq \End_A(L) \cap \End_{A'}(M')$, and the reverse inclusion is obvious.
Thus $\End_{A''}(N) = B'' = B(G, H, R) \cap B(G, R, J) = B(G, H, J)$.   

To finish the proof of the proposition, we note that 
as $S$ varies over all finite subsets of $\mb{N}$, the cycles $G = Z + \sigma^{-1}(D) - D$ obtained above 
with $D = \sum_{j \in S} Z_j$ attain every pleasantly alternating cycle supported on the $Z_i$ with $i \geq 0$.  
Explicitly, if $\sum_{i = m}^n a_i Z_i$ is pleasantly alternating with $0 \leq m \leq n$ and $a_m = a_n = 1$, then 
take $S = \{ i \geq 0 | \sum_{j \leq i} a_j = 0 \}$.
However, we can also shift indices and do the whole argument above beginning with some $\wt{G} = Z_a$ with $a \leq 0$ instead, and obtain that all rings $B(G, H, J)$ 
with $G$ pleasantly alternating and supported along the $Z_a$ with $i \geq a$ are Morita equivalent.  
In this way we get the desired result for all pleasantly alternating cycles $G$ supported along the orbit of $Z$.
\end{proof}

In the last main result of this section we give some further important properties of the rings $B(G, H, J)$.  In particular, under certain conditions these rings are simple and noetherian.  We first recall some definitions, all of which have played a role in the past study of 
birationally commutative $\mb{N}$-graded algebras, and continue to be important in the $\mb{Z}$-graded setting.
\begin{definition}
We say that an automorphism $\sigma: X \to X$ of a scheme $X$ is \emph{wild} if the only closed subsets $Z \subseteq X$ 
with $\sigma(Z) =  Z$ are $\emptyset$ and $X$.
We say that a ring $R$ with automorphism $\sigma: R \to R$ is \emph{$\sigma$-simple} if the only 
ideals $I$ of $R$ with $\sigma(I) =  I$ are $0$ and $R$.  
It is easy to check that if $R$ is a commutative noetherian ring with automorphism $\sigma$, then $R$ is $\sigma$-simple 
if and only if $\sigma: X \to X$ is wild, where $X = \spec R$ and $\sigma$ is the induced automorphism.
\end{definition}
\noindent 
The only finite type affine varieties with wild automorphisms we know are certain commutative algebraic groups with translation automorphisms; see Remark~\ref{rem:wild} below.   

\begin{definition}
Let $X$ be a variety.  A collection of distinct closed subsets $\{ W_{\alpha} \}$ of $X$ is \emph{critically dense} 
if for any proper closed subset $Y \subsetneq X$, one has $Y \cap W_{\alpha} = \emptyset$ for all but finitely many $\alpha$.
\end{definition}
\begin{definition}
A $k$-algebra $A$ is \emph{strongly noetherian} if $A \otimes_k C$ is noetherian for all noetherian commutative $k$-algebras $C$.
\end{definition}
\noindent
The strong noetherian property arises in the theory of noncommutative Hilbert schemes \cite{AZ}, where it gives a sufficient 
condition for the Hilbert schemes over an $\mb{N}$-graded ring to be projective schemes.  Though many familiar algebras 
are strongly noetherian, there are also numerous examples of rings which are noetherian but fail to be strongly noetherian.  
Whether or not the strong noetherian property holds is an important question for any ring related to noncommutative geometry.

\begin{proposition} 
\label{prop:Bsimp}
Assume the setup and notation of Definition~\ref{def:B}.
\begin{enumerate}
\item If $\{ \sigma^i(Z) | i \in \mb{Z} \}$ is critically dense in $X$, then $B(G, H, J)$ is noetherian.

\item Suppose that $\{ \sigma^i(Z) | i \in \mb{Z} \}$ is critically dense in $X$, that $R$ is a finitely generated regular $k$-algebra, and that $Z$ has codimension at least $2$ in $X$.  Then $B(G, H, J)$ is noetherian but not strongly noetherian.

\item If $\sigma$ is a wild automorphism of $X$, then $B(G, H, J)$ is a simple ring.
\end{enumerate}
\end{proposition}
\begin{proof}
(1) We follow the ideas of \cite{KRS}.  Since the noetherian property is Morita invariant, by Proposition~\ref{prop:Morita} we can pass to the case that $G = Z$.  Then $A = B_{\geq 0}$ is generated by $B_0 = R$ and $B_1 = Jt$, by Lemma~\ref{lem:B-props1}(2), and 
$B_n = J\sigma(J) \dots \sigma^{n-1}(J) t^n$ for $n \geq 1$.  

Let $I$ be a right ideal of $A$, and write $I = \sum_{n \geq 0} V_n t^n$ where each $V_n$ is an ideal of $R$.  Using $I_n A_1 \subseteq I_{n+1}$ we have 
\begin{equation}
\label{eq:V}
V_n \sigma^n(J) \subseteq V_{n+1} \subseteq  J\sigma(J) \dots \sigma^{n-1}(J) \sigma^n(J).
\end{equation}
We claim now that $I_n A_1 = I_{n+1}$, for $n \gg 0$.  The proof is essentially the same as in  \cite[Proposition 3.10]{KRS}, 
but since the notation there requires some translation we give the proof here for the convenience of the reader.
We may assume that $I \neq 0$.  Choose $r$ so that $I_r \neq 0$.  The critical density of the set $\{Z_n \}$ implies 
that $V_r  + \sigma^i(J) = R$ for all but finitely many $i$.  Thus there exists $m \geq r$ such that 
$V_r + \sigma^j(J) = R$ for all $j \geq m$.   For $n > m$ set $L_n = \sigma^m(J) \cdots \sigma^{n-1}(J)$. 
By the choice of $m$ and \eqref{eq:V}, 
$V_m$ and $\sigma^j(J)$ are also comaximal for $j \geq m$, so $V_m$ and $L_n$ are comaximal for any $n > m$.
Thus \eqref{eq:V} and induction implies that $V_m \cap L_n = V_m L_n \subseteq V_n \subseteq L_n$.  
An easy calculation now shows that $F_n = V_m + V_n$ is uniquely maximal among ideals $K$ such that $K L_n \subseteq V_n$, 
and that $F_n L_n = F_n \cap L_n = V_n$.  Since $F_n L_{n+1} = F_n L_n \sigma^n(J) \subseteq V_n \sigma^n(J) \subseteq V_{n+1}$, 
we get $F_n \subseteq F_{n+1}$ for all $n \geq m$.   Choosing $n_0$ so that $F_n = F_{n+1}$ for $n \geq n_0$, 
we see that $V_{n+1} = F_n L_{n+1} = V_n \sigma^n(J)$, or equivalently $I_n A_1 = I_{n+1}$, for $n \gg 0$, as needed.
Now since each graded piece of $I$ 
is a finitely generated $R$-module, clearly this shows that every right ideal of $A$ is finitely generated.  

A similar proof shows that every left ideal of $A$ is finitely generated, so that $A$ is noetherian.  An analogous proof also shows 
that $B_{\leq 0}$ is noetherian, or one may use Lemma~\ref{lem:anti} to reduce to the positively graded case.
Finally, it is easy to see that since $B_{\geq 0}$ and $B_{\leq 0}$ are both noetherian, then $B$ is noetherian.

(2)  This is similar to the proof of \cite[Theorem 9.2]{KRS}.   The noetherian property holds for $B(G, H, J)$ by part (1).
It is straightforward to check that the strong noetherian property is also Morita invariant, so it is enough to prove 
that $B = B(Z, H, J)$ is not strongly noetherian.  Consider $B$ as an $(B, R)$-bimodule, where $R = B_0$ acts on the right  by restriction.  Equivalently, $B$ is a left $B \otimes_k R$-module.  We claim that $B$ is not a generically flat $R$-module, in other words there is no single element $0 \neq f \in R$ such that $B_f$ is a flat $R_f$-module.   Note that $B_f$ is flat if and only if $(B_n)_f$ is for each $n \in \mb{Z}$.

Suppose that $I$ is an ideal of $R$ defining a closed subset $V = V(I)$ of codimension at least $2$, and suppose that $I_f$ is flat 
where $0 \neq f \in R$.  If $\mf{m}$ is a maximal ideal 
containing $I$, and $f \not \in \mf{m}$, then $I_{\mf{m}}$ is flat also, but $R_{\mf{m}}$ is regular local and since $\spec R_{\mf{m}}/I_{\mf{m}}$ still has codimension $2$ in $\spec R_{\mf{m}}$, the ideal $I_{\mf{m}}$ cannot be principal, a contradiction.   Thus $V(f)$ must contain $V(I)$.  
Now $B_n = [J \sigma(J) \dots \sigma^{n-1}(J)]t^n$ and by critical density any $0 \neq f$ has $V(f) \cap V(\sigma^i(J)) = \emptyset$ 
for all but finitely many $i$.  In particular, $(B_n)_f$ is not flat for $n \gg 0$ and so $B$ is not generally flat over $R$, as claimed.
 
Since $B$ is a finitely generated $B \otimes_k R$-module which is not generically flat over $R$, \cite[Theorem 0.1]{ASZ}
implies that $B \otimes_k R$ is not strongly right noetherian.  Then since $R$ is commutative affine, this implies that $B$ is not strongly right noetherian.  The same argument proves that $B$ is not strongly left noetherian. 

(3)  Since simplicity is a Morita invariant property, we may pass to the Morita equivalent ring $B = B(Z, H, J)$.  Let $A = B_{\geq 0}$.  We claim that every nonzero homgeneous ideal of $A$ contains $A_{\geq n}$ for some $n \geq 0$.  Suppose that 
$I = \bigoplus_{n \geq 0} V_n t^n$ is a nonzero ideal of $A$.  Then $B_m I_n + I_n B_m \subseteq I_{n+m}$ translates into 
$J_m \sigma^m(V_n) + V_n \sigma^{n}(J_m) \subseteq V_{n+m}$, where $J_m = J \sigma(J) \cdots \sigma^{m-1}(J)$ and 
$V_n \subseteq J_n$.   Note that for some $n_0 \geq 0$, $I_n \neq 0$ for all $n \geq n_0$.
Let $W_n$ be the closed subset of $X$ defined by $V_n$; 
so certainly $Z_0 \cup \dots \cup Z_{n-1} \subseteq W_n$, where $Z_i = \sigma^{-i}(Z)$ as usual.
For each $n \geq n_0$, let $Y_n$ be the closure of those generic points of the irreducible components of $W_n$ which are 
not contained in $Z_0 \cup \dots \cup Z_{n-1}$.  Thus $Y_n$ is the uniquely smallest closed subset of $X$ such that 
$W_n = Y_n \cup Z_0 \cup \dots \cup Z_{n-1}$.
Since $X$ has DCC on closed subsets, we may find $n \geq n_0$ such that there does not exist $i \geq n_0$ with 
$Y_i \subsetneq Y_n$.   From the equation $J_n \sigma^n(V_n) + V_n \sigma^{n}(J_n) \subseteq V_{2n}$ we get 
\[Y_{2n} \subseteq [\sigma^{-n}(Y_n) \cup Z_0 \cup \dots \cup Z_{2n-1} ]  \cap [Y_n \cup Z \cup \dots \cup Z_{2n-1}].
\]
Since each generic point of $Y_{2n}$ is not contained in $Z_0 \cup \dots \cup Z_{2n-1}$, we see that this forces
$Y_{2n} \subseteq Y_n$ and $Y_{2n} \subseteq \sigma^{-n}(Y_n)$.  By choice of $n$ we have $Y_{2n} = Y_n$, and thus 
$Y_n \subseteq \sigma^{-n}(Y_n)$.  Again since $X$ is noetherian, this gives $Y_n = \sigma^{-n}(Y_n)$.  Now since 
$\sigma$ is a wild automorphism and $Y_n \neq X$ this forces $Y_n = \emptyset$.
We now have that $R/V_n$ is supported exactly along $Z_0 \cup \dots \cup Z_{n-1}$.  Since 
$V_{2n} \supseteq V_n \sigma^{n}(J_n) + J_n \sigma^n(V_n)$, looking locally along each $Z_i$ we see that this forces
$J_{2n} \subseteq V_{2n}$.  Thus $V_{2n} = J_{2n}$ and so $B_{2n} = A_{2n} \subseteq I$.  Now since $A$ is generated in 
degrees $0$ and $1$, $A_{\geq 2n} = B_{\geq 2n} \subseteq I$, proving the claim.

Similarly, every homogeneous ideal of $B_{\leq 0}$ contains $B_{\leq -n}$ for some $n \gg 0$.  Now if $I$ is a nonzero homogeneous ideal of $B$, 
then we obtain $n > 0$ such that $B_{\geq n} + B_{\leq -n} \subseteq I$.    By Lemma~\ref{lem:B-props1}(3), $I$ is the unit ideal of $B$.
This proves that $B$ is \emph{graded simple}, that is, that $B$ has no proper homogeneous ideals.  Now 
$Q = Q_{\rm gr}(B) \cong K[t, t^{-1}; \sigma]$, where $K$ is the quotient field of $R$.  The setup in Definition~\ref{def:B} certainly 
forces $\sigma$ to be of infinite order.  Then $Q$ is a simple ring, and it easily follows that $B$ is simple if and only if it is graded simple
(see \cite[Lemma 2.6(2)]{BRS}), which finishes the proof.
\end{proof}

\begin{example}
As already mentioned, many examples of noetherian but not strongly noetherian $\mb{N}$-graded 
examples are known; see for example \cite{KRS}.   It is interesting to note that in the $\mb{Z}$-graded setting one can even get such examples which are simple.  To give an easy explicit example, take the torus $X = \spec R$ where $R = k[x_1^{\pm 1}, x_2^{\pm 1}]$ with automorphism $\sigma(x_1) = p_1x_1$ and $\sigma(x_2) = p_2 x_2$ for constants $p_1, p_2$ generating a free abelian subgroup of $k$.  The automorphism $\sigma$ is wild by Lemma~\ref{T-lem} below.  Then let $Z = q$ be any closed point of $X$, defined by the maximal ideal $\mf{m}$, say, and take $B = B(Z, H, J)$ with $H = R, J = \mf{m}$.   The $\sigma$-orbit of $q$ is critically dense by an affine version of the argument in \cite[Theorem 12.3]{Ro}.
Now apply Proposition~\ref{prop:Bsimp} to see that $B$ is simple and noetherian but not strongly noetherian.
\end{example}

\section{Properties of $\mb{Z}$-graded algebras}
\label{simple-sec}

In this section, we study some general results about the properties of $\mb{Z}$-graded algebras. 
Then in the next section, we will use these results to classify simple, birationally commutative $\mb{Z}$-graded algebras under 
some further assumptions.

We first make a general comment about the possibility of graded pieces which are $0$.
Suppose that $A = \bigoplus_{n \in \mb{Z}} A_n$ is a $\mb{Z}$-graded $k$-algebra, and let $S =
\{n \in \mb{Z} | A_n \neq 0 \}$.   We exclude the trivial case $S = \{ 0 \}$ where the grading is irrelevant.
If $\gcd(S) = d > 1$, then we may study the $d$th Veronese ring
$A^{(d)} = \bigoplus_{n \geq 0} A_{nd}$ instead with no loss of information.  Thus 
we assume that $d = 1$.  Now if $A$ is a domain, then $S$ is also a sub-semigroup of $\mb{Z}$.
Then it is trivial to prove that exactly one of the following happens: $A$ is $\mb{N}$-graded with $A_n \neq 0$ for all $n \gg 0$; (ii) $A$ is $-\mb{N}$-graded with $A_n \neq 0$ for all $n \ll 0$; or (iii) $A_n
\neq 0$ for all $n \in \mb{Z}$.   If $A$ is simple, then case (iii) is forced.  In this paper we are primarily 
interested in $\mb{Z}$-graded simple domains $A$, so it is reasonable to always assume that $A_n \neq 0$ for all $n \in \mb{Z}$.

Recall that the graded ring of fractions of a $\mb{Z}$-graded Ore domain $A$ has the form $D[t, t^{-1}; \sigma]$ for a division ring $D$, 
and that $A$ is \emph{birationally commutative} if $D$ is a field.  In the next result, we study how some properties of the division ring $D$ are related to the properties of $A$.  We assume that the reader is familiar with the basic properties of Gelfand-Kirillov (GK) dimension, for 
which \cite{KL} is the standard reference.  In particular, 
we will use the convenient result that a domain of finite GK-dimension is automatically an Ore domain \cite[Proposition 4.13]{KL}.
\begin{lemma}
\label{D-fg-lem} 
Let $A$ be a  $\mb{Z}$-graded $k$-algebra which is an Ore domain with $A_n \neq 0$ for all $n \in \mb{Z}$.  Let $Q  = Q_{{\rm gr}}(A) \cong D[t, t^{-1}; \sigma]$. 
\begin{enumerate}
\item $A_0$ is a Ore domain, with quotient division ring $D$.
\item If $A$ is a finitely generated $k$-algebra with $\GK A < \infty$, then $D$ is finitely generated as a division algebra over $k$.
Moreover, if $\GK A = 2$ then $D$ is a finitely generated field extension of transcendence degree $1$ over $k$.
\end{enumerate}
\end{lemma}
\begin{proof}
(1) It is clear that $A_0$ is an Ore domain, by restricting the Ore condition to degree $0$ elements.
An arbitrary element of $D$ is of the form $x y^{-1}$ for some $x, y \in A_n$, some $n$.  
Choosing $0 \neq z \in A_{-n}$ then $xy^{-1} = (xz)(yz)^{-1}$ is an element of 
the Ore quotient ring $Q(A_0)$ of $A_0$.  So $A_0 \subseteq D \subseteq Q(A_0)$.  Conversely, $Q(A_0) \subseteq Q_{{\rm gr}}(A)_0 = D$ 
is obvious.  

(2)  We reduce to the $\mb{N}$-graded case, where the result is known.
Consider a finite homogeneous generating set $\{ x_1, \dots,
x_r \}$ for $A$ as a $k$-algebra.  In $Q$, we may write $x_i = a_i b^{-1}$ for some $a_i, b \in A$ of positive degree.
Now let $A' = k \langle a_1, \dots, a_r, b \rangle$,
which is a connected $\mb{N}$-graded subalgebra of $A$.  
The ring $A'$ is also an Ore domain since $\GK A' < \GK A  < \infty$.
Clearly by construction $Q_{\rm gr}(A) = Q_{\rm gr}(A')$.   Then \cite[Theorem 1.15]{AS}
proves that $D = Q_{\rm gr}(A')_0$ is finitely generated as a
division algebra over $k$.

Now suppose that $\GK A = 2$.  Then $\GK A' \leq 2$.  
It is well-known that a connected graded domain $B$ with 
$\GK B = 1$ over an algebraically closed field $k$ must have $Q_{\rm gr}(B) \cong k[t, t^{-1}]$.    This shows that 
$\GK A' = 1$ is impossible.  Since no 
algebra has GK-dimension between $1$ and $2$ \cite[Theorem 2.5]{KL}, $\GK A' = 2$ also.  
Then $D = Q_{\rm gr}(A')_0$ is a field of transcendence degree $1$ over $k$, by \cite[Theorem 0.1]{AS}.
\end{proof}

The hypothesis that a $\mb{Z}$-graded algebra $A$ be finitely generated as a $k$-algebra is awkward to work with in some ways.   
For example, it does not imply that $A_0$ is finitely generated as a $k$-algebra, as the following standard example shows.
\begin{example}
\label{ex:weyl-loc}
Let $A$ be the Weyl algebra $k \langle x, y \rangle/(xy-yx -1)$ 
and let $\Omega = \{ xy-m | m \in \mb{Z} \}$.  Then $\Omega$ is an Ore set in $A$ and $A\Omega^{-1}$ is still finitely generated as an algebra; 
moreover, $\GK A  = 2$ while $\GK A\Omega^{-1} = 3$  \cite[Example 4.11]{KL}.  Note that if $A$ is $\mb{Z}$-graded as usual with $\deg x = 1, \deg y = -1$, 
then $B = A \Omega^{-1}$ is still $\mb{Z}$-graded, and $B \subseteq Q_{\rm gr}(A) \cong k(z)[x, x^{-1}]$ where $z = xy$.  However, $B_0 \subseteq k(z)$ 
is not a finitely generated $k$-algebra, since $(z-m)^{-1} \in B_0$ for all $m \in \mb{Z}$, whereas a finitely generated subalgebra of $k(z)$ must consist 
of functions with poles coming from a fixed finite set.
\end{example}

Part of our technique for studying $\mb{Z}$-graded algebras $A$ is to first consider the $\mb{N}$-graded part $A_{\geq 0}$, and the example 
above shows that we cannot expect this to be a finitely generated algebra just because $A$ is.  Thus it is convenient to 
work with the following weaker hypothesis which is still sufficient for our applications.
\begin{definition}
Let $A = \bigoplus_{n \in \mb{Z}} A_n$ be any $\mb{Z}$-graded $k$-algebra.  We say that $A$ is \emph{quasi-finitely generated} if there is a finite set of degrees $S \subseteq \mb{Z}$ such that $\{ A_i | i \in S \}$ generates
$A$ as a $k$-algebra.
\end{definition}
It is obvious that a finitely generated $\mb{Z}$-graded algebra is quasi-finitely generated, but in contrast to Example~\ref{ex:weyl-loc} 
we have the following.
\begin{lemma}
\label{fg-lem} 
\label{lem:fg}
Let $B$ be a $\mb{Z}$-graded quasi-finitely generated $k$-algebra.
Then $B_{\geq 0}$ is also quasi-finitely generated, or equivalently there is $r \geq 1$
such that $B_n = \sum_{i=1}^{r} B_i B_{n-i}$ for $n > r$.
\end{lemma}
\begin{proof}
We may assume that $B$ is generated as a $k$-algebra by $\{ B_i | -r \leq i \leq r \}$, some $r \geq 1$.
For $n > r$, an arbitrary element of $B_n$ is a sum of words $w = x_1x_2 \dots x_m$ 
where each $x_i$ has degree $d_i$ with $|d_i| \leq r$ and $\sum d_i = n$.
Define $w_j =x_1 \dots x_j$ for each $j$.  There is a smallest $j \geq 1$ such that $\deg
w_j > 0$, in which case clearly $1 \leq \deg w_j \leq r$.  Also, $j \neq m$, 
since $\deg w_m = \deg w = n> r$.
Now set $w' = w_j$ and $w'' = x_{j+1} \dots x_m$.  Then 
$w = w' w'' \in B_i B_{n-i}$ for $i = \deg w_j$.   Thus $B_n = \sum_{i=1}^{r} B_i B_{n-i}$ for all $n > r$.  
Equivalently, $B_{\geq 0}$ is generated as an algebra by $\{ B_i | 0 \leq i \leq r \}$.
\end{proof}

Next, we study some consequences of assuming that $\mb{Z}$-graded algebra is simple
and birationally commutative.

\begin{lemma}
\label{lem:basics}
Let $A$ be a $\mb{Z}$-graded $k$-algebra which is a simple Ore domain with $Q_{\rm gr}(A) \cong K[t, t^{-1}; \sigma]$ for some field $K$.  Assume that $R = A_0$ is noetherian, and let $C \subseteq K$ be the $k$-algebra generated by $\{ \sigma^i(R) | i \in \mb{Z} \}$.
\begin{enumerate}
\item $C$ is $\sigma$-simple.
\item $R \subseteq C$ is an integral extension of rings, and if it is a finite extension, then $R = C$.
\item Suppose that the integral closure of $R$ is a finite  $R$-module, and that the singular locus of $X = \spec R$ is a proper closed subset of $X$.  Then $R = C$ and $R$ is a regular ring.
\end{enumerate}
\end{lemma}
\begin{proof}
(1) Clearly $\sigma$ restricts to an automorphism of $C$. 
Write $A = \bigoplus_{n \in \mb{Z}} V_i t^i$ with $V_i \subseteq K$, and let 
$A' = \bigoplus_{n \in \mb{Z}} C V_i t^i$, which is easily checked to be a subring of $K[t, t^{-1}; \sigma]$ also. 
Since $A$ is simple and $A \subseteq A' \subseteq Q_{\rm gr}(A)$, an easy argument shows that  $A'$ is also simple.

Now if $I$ is an ideal of $C$ with $\sigma(I) = I$, then $IA' = \bigoplus IV_n t^n$ is a an ideal of $A'$.  
Since $(IA')_0 = I$, this ideal is proper in $A'$ if $I$ is proper in $C$.  Since $A'$ is a 
simple ring, $I = 0$ or $I = C$.  Then $C$ is $\sigma$-simple. 

(2)  Again we write $A = \bigoplus_{n \in \mb{Z}} V_i t^i$ with $V_i \subseteq K$.
Note that picking any nonzero elements $v_i \in V_i$, $w_i \in V_{-i}$,
then $v_i R t^i w_i t^{-i} = v_i \sigma^i(w_i) \sigma^i(R) \subseteq R$.  
This shows, since $R$ is noetherian, that the algebra $\sigma^i(R) R$ is a finite $R$-module,
and hence $R \subseteq \sigma^i(R)R$ is an integral extension of rings.  Since $C$ is generated 
by elements integral over $R$, $R \subseteq C$ is an integral extension.  

Now suppose that $R \subseteq C$ is a finite extension, so 
there is $0 \neq b \in R$ such that $bC \subseteq R$.  
Defining again the ring $A' = \bigoplus_{n \in \mb{Z}} C V_i t^i$, we have $bA' \subseteq A$.   Then 
$I = \operatorname{r.ann}_A(A'/A) \neq 0$, forcing $I = A$ and $A = A'$.
In particular, $R = C$.

(3) Because the integral closure of $R$ is a finite $R$-module, so is the integral extension $C$ of $R$.  Thus $R= C$ by part (2).  Then $\sigma(R) = R$ and $R$ is $\sigma$-simple by part (1).  By hypothesis, the set of points $S \subseteq X = \spec R$ where $R$ is singular 
(that is, those primes $p$ such that $R_p$ is not regular local) is a proper closed subset in the Zariski topology.
Since $\sigma$ is an automorphism of $R$, $\sigma(S) \subseteq S$.  Since $\sigma: X \to X$ 
is wild, $S = \emptyset$ and $R$ is a regular ring.
\end{proof}

The hypotheses on $R = A_0$ in part (3) of the lemma above are very weak and hold, for example, for all excellent rings, a class containing most of the commutative noetherian rings one encounters in practice, in particular finitely generated $k$-algebras \cite[32.B, 33.H, 34.A]{Ma}.
Thus, morally the lemma says that the zeroth degree piece of a simple birationally commutative 
$\mb{Z}$-graded algebra $A$ ought to be regular.  However, we are unable to rule out the possibility that there exist such examples $A$ where 
$A_0$ has bizarre properties.

With the previous lemma as justification, in the remainder of this section we study the further properties of $\mb{Z}$-graded birationally commutative algebras whose degree zero piece is regular.   We first need to review some properties of modules over a regular ring, as well as some definitions related to divisors on the corresponding affine scheme.
Let $R$ be a regular noetherian commutative domain with field of fractions $K$, and let $X = \spec R$.
For any finitely generated $R$-submodule
$M$ of $K$, write $M^* = \Hom_R(M,R)$, which as in Section~\ref{sec:somerings} we identify with  
$\{x \in K | xM \subseteq R\}$.   Then $M \subseteq M^{**} \subseteq K$, and $M$ is called 
\emph{reflexive} if $M = M^{**}$.  In general the module $M^{**}$ is reflexive and is called the \emph{reflexive hull} of $M$.  
For convenience of notation we write $\wt{M}$ for $M^{**}$ from now on.

Recall that a (Weil) divisor on $X$ is a formal $\mb{Z}$-linear 
combination of irreducible closed subsets of codimension 1 in $X$.  We say a divisor is effective if 
all of its coefficients are nonnegative, and write $D \geq E$ if $D - E$ is effective.  This  partial order 
determines max and min operations on divisors similarly as we defined for cycles in Definition~\ref{def:cycle}.
Every $f \in K$ has an associated \emph{principal divisor}
$(f) =  \sum_{Z} \nu_Z(f) Z$, where we sum over all codimension-1 irreducible closed subsets $Z$ of $X$, and where $\nu_Z$
is the valuation on $K$ measuring the order of vanishing of $f$ along $Z$.  The Picard group $\Pic(X)$ is the group of 
all divisors modulo principal divisors.  There is a one-to-one correspondence between finitely generated locally free $R$-submodules of $K$ and Weil divisors on $X$.  Explicitly, given a finitely generated locally free $R$-submodule $M \subseteq K$, $M = \mc{O}_X(D)$ where 
$D$ is the unique smallest divisor such that $(f) + D$ is effective for all $f \in M$.
If $\sigma: R \to R$ is an automorphism, inducing $\sigma: X \to X$ which acts on divisors in the obvious way, 
then $\sigma(\mc{O}_X(D)) = \mc{O}_X(\sigma^{-1}(D))$.
More detail about all of the definitions above can be found in \cite[Section II.6]{Ha}.

We recall the following standard facts. 
\begin{lemma}
\label{lem:regfacts}
Let $R$ be a regular noetherian commutative domain with field of fractions $K$, let $X = \spec R$, and 
let $M, N \subseteq K$ be  finitely generated $R$-submodules.  
\begin{enumerate}
\item The following are equivalent: (i) $M$ is invertible; (ii) $M$ is locally principal; and 
(iii) $M$ is reflexive.
\item $\wt{M} \wt{N} = \wt{MN}$. 
\item If $M = \mc{O}_X(B)$ and $N = \mc{O}_X(C)$, then $MN = \wt{MN} = \mc{O}_X(B + C)$ and $\wt{M + N} = \mc{O}_X(\max(B, C))$.
\end{enumerate}
\end{lemma}
\begin{proof}
(1) The implications $(i) \Longleftrightarrow (ii)$ and $(ii) \implies (iii)$ are standard and don't even require $R$ to be regular.  
The less obvious implication $(iii) \implies (ii)$ can be found in \cite[Proposition 1.9]{Ha80}.

(2) This can be proved locally, so we can assume that $R$ is regular local and hence a UFD.   By 
multiplying by a suitable element of $K$, we can assume that $M$ and $N$ are ideals of $R$.
Then it is easy to see that if $M = a_1R + \dots + a_n R$, then $\wt{M} = bR$ where $b = \gcd(a_1, \dots, a_n)$.  The result easily follows from this.

(3)  The first formula follows from part (2).  The argument in part (2) also shows that the reflexive hull of $M$ is the unique smallest locally principal submodule of $K$ containing $M$, from which the 
second formula follows.
\end{proof}

The following definition and lemma adapt to the affine case a concept from \cite{AS}, which was studied in that paper for sequences of 
divisors on projective curves only.
\begin{definition}
\label{def:divseq}
Let $E_0, E_1, E_2, \dots$ be a sequence of divisors on $X = \spec R$, where $R$ is a regular domain 
with automorphism $\sigma: R \to R$ incuding $\sigma: X \to X$.  We say that $\{E_i \}$ is a \emph{$\sigma$-divisor sequence}
if $E_0 = 0$, $E_i + \sigma^{-i}(E_j) \leq E_{i + j}$ for all $i, j \geq 0$, and
$E_n = \max_{i=1}^r (E_i + \sigma^{-i}E_{n-i})$ for all $n > r$, some $r \geq 1$.
\end{definition}

\begin{lemma}
\label{lem:tildalg}
Let $A = \bigoplus_{n \geq 0} W_n t^n$ be a subalgebra of $K[t, t^{-1}; \sigma]$, where $K$ is the field of fractions of the regular $k$-algebra $R = A_0$ with automorphism $\sigma$.  Then $\wt{A} = \bigoplus_{n \in \mb{Z}} \wt{W_n} t^n$ is also a subalgebra of $K[t, t^{-1}; \sigma]$.
If $A_{\geq 0}$ is quasi-finitely generated and we write $\wt{W_n} = \mc{O}_X(D_n)$ for each $n$, then $D_0, D_1, \dots$ is a $\sigma$-divisor sequence.
\end{lemma}
\begin{proof}
It is immediate that $\wt{A}$ is a subalgebra, since $W_n \sigma^n(W_m) \subseteq W_{n +m}$ implies 
$\wt{W_n} \wt{\sigma^n(W_m)} \subseteq \wt{W_{n+m}}$ for all $n,m$ by Lemma~\ref{lem:regfacts}(2).
Equivalently, $D_n + \sigma^{-n}(D_m) \leq D_{m+n}$ for all $m, n \in \mb{Z}$.
Since $A_{\geq 0}$ is quasi-finitely generated, we have $r \geq 1$ such that $\sum_{i = 1}^r W_i \sigma^i(W_{n-i}) = W_n$ for all $n > r$.
By Lemma~\ref{lem:regfacts}(3) this is equivalent to $D_n = \max_{i=1}^r (D_i + \sigma^{-i}D_{n-i})$ for all $n > r$.
Finally, obviously $D_0 = 0$ since $\wt{W_0} = R$.  Thus $D_0, D_1, \dots$ is a $\sigma$-divisor sequence.  
\end{proof}
   
The basic combinatorial analysis of $\sigma$-divisor sequences, which was worked out in \cite{AS}, goes through in our setting 
as follows.
\begin{lemma}
\label{comb-lem}
\label{lem:AS-comb}
Let $X = \spec R$ for a commutative noetherian regular $k$-algebra $R$ with automorphism $\sigma: R \to R$.
Let $E_0, E_1, E_2, \dots$ be a $\sigma$-divisor sequence on $X$, where all irreducible divisors in the support of each $E_i$ lie on infinite $\sigma$-orbits.  Then there are divisors $G$ and $\Omega$, where $\Omega$ is effective, such that $E_n = G_n - \Omega$  for all $n \gg 0$, where $G_n = G + \sigma^{-1}(G) + \dots + \sigma^{-n+1}(G)$.
\end{lemma}
\begin{proof}
In \cite{AS} the divisors in a $\sigma$-divisor sequence are assumed to be effective, and so we first 
discuss how to remove this restriction.   Let $r$ be the integer such that $E_n = \max_{i=1}^r (E_i + \sigma^{-i}(E_{n-i}))$ for all $n > r$.
Choose an effective divisor $H$ large enough so that $E_i + H$ is also effective for $1 \leq i \leq r$.   Putting 
$H_n = H + \sigma^{-1}(H) + \dots + \sigma^{-n+1}(H)$ for all $n \geq 0$, then $H_i \geq H$ for $i \geq 1$ and so 
$E'_i = E_i + H_i$ is effective for $1 \leq i \leq r$.  Then setting $E'_n = E_n + H_n$, clearly 
 $E'_0, E'_1, \dots$ is still a $\sigma$-divisor sequence, for the same $r$, and $E'_n$ is effective 
for all $n \geq 0$ by induction.   Now if we prove that the lemma holds 
for the sequence $E'_i$, we obtain $G'$ and $\Omega  \geq 0$ such that $E'_n = G'_n - \Omega$ for $n \gg 0$.  Then 
$E_n = G_n - \Omega$ for all $n \gg 0$, where $G = G' - H$.

So now we may assume that the $E_i$ are effective.   Then \cite[Lemma 2.17]{AS} provides an effective 
divisor $\Omega$ such that the sequence $D_n = E_n + \Omega$ satisfies $D_i + \sigma^{-j}(D_i) = D_{i+j}$ for all $i,j \gg 0$. 
Note that while \cite[Lemma 2.17]{AS} is stated only for divisors on a projective curve, all that is really used is the combinatorics of the coefficients of these 
divisors that results from the $\sigma$-divisor sequence condition, together with the assumption that all of the 
irreducible divisors occurring lie on infinite $\sigma$-orbits.   
Finally, we claim that $G = D_{n+1} - \sigma^{-1}(D_n)$ is independent of the choice of $n \gg 0$, and that with this choice 
of $G$ we have $G_n = G + \sigma^{-1}(G) + \dots + \sigma^{-n+1}(G) = D_n$ for all $n \gg 0$.  This is the same 
as what is proved in \cite[Lemma 5.8(i)(ii)]{AS}.
\end{proof}

In the last results of this section, we prove that up to a fairly trivial kind of adjustment, in the study of 
birationally commutative $\mb{Z}$-graded algebras $A \subseteq K[t, t^{-1}; \sigma]$ where $A_0 = R$ is regular
and $\sigma(R) = R$, we can reduce to the convenient case that $A \subseteq R[t, t^{-1}, \sigma]$.
\begin{lemma}
\label{lem:triv-change}
Let $R$ be a commutative noetherian regular ring with automorphism $\sigma$, and let $K$ be its fraction field.
Let $A = \bigoplus_{n \in \mb{Z}} W_n t^n$ be a subalgebra of $K[t, t^{-1}; \sigma]$, where $A_0 = R$ and $W_n \neq 0$ for all $n \in \mb{Z}$. Choose  any locally principal $R$-submodule $M \subseteq K$.
Write $M = \mc{O}_X(D)$, and let $M_n = \mc{O}_X(D_n)$ for each $n \in \mb{Z}$, where $D_n$ is defined in terms of $D$ as in 
Definition~\ref{def:Gn}.

Then $B = \bigoplus_{n \in \mb{Z}} M_n W_n t^n$ is also a subalgebra of $K[t, t^{-1}; \sigma]$, and there 
are equivalences of categories $B \lGr  \sim A \lGr$ and $\rGr B \sim \rGr A$.  
Moreover, $B$ is simple if and only if $A$ is.  If $R$ is a UFD, then $A$ and $B$ are even isomorphic.  
\end{lemma}
\begin{proof}
Since $D_m + \sigma^{-m}(D_n) = D_{m+n}$ for all $m,n \in \mb{Z}$ by Lemma~\ref{lem:easyGfacts}, 
we have $M_m \sigma^m(M_n) = M_{m+n}$ for all $m,n \in \mb{Z}$.  Then it is clear that $B$ is a subalgebra of 
$K[t, t^{-1}; \sigma]$.  There is a functor which sends a $\mb{Z}$-graded left $A$-module $\bigoplus_{n \in \mb{Z}} V_n$ to the graded left $B$-module $\bigoplus_{n \in \mb{Z}} (M_n \otimes_R V_n)$, with action defined by 
$x a (y \otimes v) = x \sigma^m(y) \otimes av$, where $x \in M_m, a \in A_m, y \in M_n, v \in V_n$.  This functor is 
easily checked to give an equivalence of graded categories $A \lGr \to B \lGr$.  The equivalence 
of categories of graded right modules is similarly routine.

If $\sigma$ has finite order, then $K[t, t^{-1}; \sigma]$ is finite over its center and thus a PI ring; so neither $A$ nor 
$B$ can be simple since a simple PI domain is a division ring.   On the other hand, if $\sigma$ has infinite order, then the ring $K[t, t^{-1}; \sigma]$ is simple, and the same argument we have already seen in Proposition~\ref{prop:Bsimp}(3) shows that $A$ is simple if and only if it is graded simple, and similarly for $B$.   But it is trivial to check that $A$ is graded simple if and only if $B$ is.

Finally, if $R$ is a UFD, then $M$ is principal, say $M = x R$ for $x \in K$.   Defining $x_0 = 1$, $x_n = x \sigma(x) \dots \sigma^{n-1}(x)$ for $n \geq 1$, and $x_{-n} = [\sigma^{-1}(x) \dots \sigma^{-n}(x)]^{-1}$ for $n \geq 1$, the map $\phi: A \to B$ defined by $\phi(a) = x_n a$ for 
$a \in A_n$ is easily checked to be an isomorphism.
\end{proof}

\begin{definition}
Given $A$ and $B$ as in the previous lemma, we say that $B$ is a \emph{$\Pic(X)$-twist} of $A$.
\end{definition}
\noindent It is not hard to see that $\Pic(X)$-twists need not be isomorphic or even Morita equivalent in general.  Still, a $\Pic(X)$-twist is a simple operation which preserves many properties of a graded ring.   
\begin{lemma}
\label{lem:trivchange}
Let $R$ be a commutative noetherian regular ring with automorphism $\sigma$, and let $K$ be its fraction field.
Let $A = \bigoplus_{n \in \mb{Z}} W_n t^n$ be a subalgebra of $K[t, t^{-1}; \sigma]$, where $A_0 = R$ and $W_n \neq 0$ for all $n \in \mb{Z}$.
Assume that $A_{\geq 0}$ is quasi-finitely generated.  Then 
there is an invertible $R$-module $M \subseteq K$ such that the $\Pic(X)$-twist
$B = \bigoplus_{n} M_n W_n t^n$ satisfies $B \subseteq R[t, t^{-1}; \sigma]$.
\end{lemma}
\begin{proof}
Let $\wt{A} = \bigoplus_n \wt{W}_n t^n$ as in  Lemma~\ref{lem:tildalg}.  Writing $\wt{W}_n = \mc{O}_X(D_n)$, 
by that same result $D_0, D_1, \dots$ is a $\sigma$-divisor sequence.   By Lemma~\ref{lem:AS-comb},
$D_n = G_n - \Omega$ for all $n \gg 0$, for some effective $\Omega$ and $G_n = G + \sigma^{-1}(G) + \dots + \sigma^{-n+1}(G)$.  
Choose $M = \mc{O}_X(-G) = \mc{O}_X(G)^{-1}$, and let $B = \bigoplus_{n} M_n W_n t^n$ 
and $\wt{B} = \bigoplus_{n} M_n \wt{W}_n t^n$.   Then $\wt{B}_n = \mc{O}_X(-\Omega)t^n$ for all $n \gg 0$, 
where $I = \mc{O}_X(-\Omega)$ is some ideal of $R$.  

Since $\wt{B}$ is an algebra, for any $m \in \mb{Z}$ and $n \gg 0$ we have $\wt{B}_n \wt{B}_m \subseteq \wt{B}_{n + m}$,  and thus $I \sigma^n(M_m \wt{W}_m) \subseteq I$.  But $\Hom_R(I,I) = R$ since $I$ is invertible.  Thus 
$\sigma^n(M_m \wt{W}_m) \subseteq R$ and hence $M_m \wt{W}_m \subseteq R$.  
We conclude that $B \subseteq \wt{B} \subseteq R[t, t^{-1}; \sigma]$.
\end{proof}

\section{Classification of simple birationally commutative $\mb{Z}$-graded rings}
\label{sec:classify}

Starting in this section, we work towards a kind of converse to the results of the Section~\ref{sec:somerings}. Namely,  
we aim to find rather general hypotheses on birationally commutative simple $\mb{Z}$-graded algebras 
under which we can classify them, and more specifically show that they are closely related to the rings $B(G, H, J)$.
The following main hypothesis for this section contains the most general conditions under which 
we are able to prove our classification theorem.
\begin{hypothesis}
\label{hyp:main}
\label{UFD-hyp}
Let $k$ be algebraically closed.  Let $A$ be a simple, quasi-finitely generated $\mb{Z}$-graded $k$-algebra which is an Ore domain 
with $A_i \neq 0$ for all $i \in \mb{Z}$, and such that $Q_{\operatorname{gr}}(A) = K[t, t^{-1}; \sigma]$ for some field $K$ with automorphism $\sigma$.  Assume that either $\cha k = 0$ or that $\trdeg(K/k) = 1$.  Assume further that $R = A_0$ is a noetherian $k$-algebra such that (i) the integral closure of $R$ is a finite  $R$-module; and (ii) the singular locus of $X = \spec R$ is a proper closed subset of $X$.  
\end{hypothesis}
\noindent  
\begin{remark}
We comment on some of the conditions in the hypothesis above.
As remarked after Lemma~\ref{lem:basics}, conditions (i) and (ii) hold for most reasonable commutative noetherian rings.
It is immediate from Lemma~\ref{lem:basics}(3) that under the hypothesis above, in fact $R$ must be regular, $\sigma(R) = R$, and 
$R$ is $\sigma$-simple.   When we assume Hypothesis~\ref{hyp:main} we will use these facts without further comment.

The assumption that $A_0$ is noetherian is natural; it is easy to see that if $A$ is noetherian then so is $A_0$, so we might as well make 
the weaker assumption.  Finally, the assumption that $\cha k = 0$ or $\trdeg(K/k) = 1$ is made to  overcome a technical obstacle in the proof below, but we suspect that the main classification theorem is true without it.
\end{remark}

As we saw in the previous section, given an algebra $A$ satisfying Hypothesis~\ref{hyp:main}, then after a $\Pic(X)$-twist we can 
assume that $A \subseteq R[t, t^{-1}; \sigma]$, and we will do so in the analysis in the rest of the section.

The following result restates the hypothesis that $A$ is simple in a more convenient form.
\begin{lemma}
\label{simple-crit-lem1}
Let $A  = \bigoplus_{n \in \mb{Z}} I_n t^n \subseteq R[t, t^{-1}; \sigma]$ satisfy Hypothesis~\ref{hyp:main}.
Then for all $n \geq 1$, we have
\[
\sum_{i \geq n} \big(A_iA_{-i} + A_{-i}A_i \big)= \sum_{i \geq n} \big(I_i \sigma^i(I_{-i}) + I_{-i} \sigma^{-i}(I_i)\big) =  R.
\]
\end{lemma}
\begin{proof}
Fix $n \geq 1$ and
consider the two-sided ideal $J$ of $A$ generated by $A_{\leq -2n} \oplus A_{\geq 2n}$.
Since $A$ is simple, $J$ is the unit ideal of $A$, so $J_0 = R$.
Now $J_0$ is spanned by products $A_iA_kA_j$ with
$i, j \in \mb{Z}$, $|k| \geq 2n$, where $i + k + j = 0$.  In
particular, in any such product $A_iA_kA_j$,
since $|k| \geq 2n$, then either $|i| \geq n$ or $|j| \geq n$.
So either $(A_iA_k)A_j \subseteq A_{i+k}A_j$ with
$|i+k| = |j| \geq n$ or else $A_i(A_kA_j) \subseteq A_iA_{j+k}$ with $|i| = |j+k| \geq n$.  We conclude that
\[
R = J_0 \subseteq  \sum_{i \geq n} \big(A_iA_{-i} + A_{-i}A_i \big),
\]
which implies the result.
\end{proof}

The next idea is to study the support of $R/I_n$ for a ring $A = \bigoplus_{n \in \mb{Z}} I_n t^n \subseteq R[t, t^{-1}; \sigma]$, 
focusing on one $\sigma$-orbit of points at a time.    This requires us to set up some notation to track multiplicities of vanishing.
\begin{definition}
Let $R$ be a commutative regular $k$-algebra and let $p$ be a (not necessarily closed) point of the scheme $X = \spec R$, that is, a prime ideal of $R$.  Let $R_p$ be the local ring at $p$, with maximal ideal $p_p$.
Given $f \in R$, we define the multiplicity of vanishing of $f$ along $p$ as 
$m_p(f) = \max\{ m \geq 0 | f \in (p_p)^m \}$.  Similarly, for an ideal $I$ of $R$ its multiplicity of vanishing is 
defined to be $m_p(I) = \max\{m \geq 0 | I \subseteq (p_p)^m \}$.
\end{definition}
\noindent 
Note that if $R$ is regular, then $R_p$ is a regular local ring and its associated graded ring  $\bigoplus_{n \geq 0} (p_p)^n/(p_p)^{n+1}$
is isomorphic to a polynomial ring over the residue field \cite[Prop. 2.2.5]{BH}.  In particular, it is a domain and so we see that 
$m_p(fg) = m_p(f) + m_p(g)$ for any $f, g \in R$ and $m_p(IJ) = m_p(I) + m_p(J)$ for any ideals $I, J$ of $R$.  It is also easy to see  
that $m_p(I + J) = \min(m_p(I), m_p(J))$ for ideals $I,J$ of $R$.   Note that $m_p(I) \neq 0$ is equivalent to $I \subseteq p$.

\begin{lemma}
\label{lem:critdense}
Let $\sigma: R \to R$ be an automorphism of a noetherian $k$-algebra which is a domain, let $X = \spec R$, and assume that $R$ is $\sigma$-simple (or equivalently that $\sigma: X \to X$ is wild).  Suppose further that $\cha k = 0$ or that $\dim X = 1$.  
Then given an ideal $0 \neq I$ of $R$ and a point $p \in X$ other than the generic point, $m_{\sigma^{i}(p)}(I) \neq 0$ holds for at most finitely many $i \in \mb{Z}$.  
\end{lemma}
\begin{proof}
Note that if $q$ is a point in the closure of $p$, then $m_{p}(I) \neq 0$ implies $m_{q}(I) \neq 0$.  Thus 
it sufices to replace $p$ by any point in its closure, so we may assume that $p$ is a closed point.

Because $\sigma$ is wild, the  Zariski closure of the orbit $\{ \sigma^i(p) | i \in \mb{Z} \}$ is all of $X$.  
Then since the orbit of $p$ is dense, it is critically dense assuming $\cha k = 0$  \cite[Corollary 5.1]{Be1} (see also \cite{Be2}).  In other words, for any proper closed subset $Z$ of $X$, $\{ i \in \mb{Z} | \sigma^i(p) \in Z \}$ is finite.   Applying this to the subset $Z$ defined by the ideal $I$ gives the required result in characteristic $0$.  If instead $\dim X = 1$, then any $0 \neq I$ vanishes at finitely many points in $X$ anyway 
and the result is trivial.
\end{proof}

Recall the notion of cycle along the orbit of a closed subset from Definition~\ref{def:cycle}.
\begin{definition}
Let $A  = \bigoplus_{n \in \mb{Z}} I_n t^n \subseteq R[t, t^{-1}; \sigma]$ satisfy Hypothesis~\ref{UFD-hyp}.   Given any point $p \in X$ (except the generic point) whose closure is $Z$, we define for each $n \in \mb{Z}$ the numbers $f_{n,i} = m_{\sigma^{-i}(p)}(I_n)$ and the cycle $F_n = \sum_{i \in \mb{Z}}  f_{n,i} Z_i$.   We call the sequence $\{ F_n | n \in \mb{Z} \}$ the \emph{support cycle sequence} along the orbit of $p$.  Note that each $F_n$ is a well-defined cycle 
(that is, only finitely many coefficients are nonzero) because of Lemma~\ref{lem:critdense}.   We also 
define the \emph{reduced support cycle sequence} along the orbit of $p$ to be $\{ E_n | n \in \mb{Z} \}$, where $E_n = \sum e_{n,i} Z_i$ 
with $e_{n,i} = \min(1, f_{n,i})$ for each $i$ and $n$.  
\end{definition}

The next result is the main technical underpinning of our classification theorem.
It shows that the support cycle sequence along an orbit satisfies a combinatorial property which is dual to the concept of $\sigma$-divisor sequence, and uses this and the earlier criterion for simplicity to show that support cycle sequences must be of a very 
restricted form.  We do not know if the following result still holds in characteristic $p$, since Lemma~\ref{lem:critdense} may fail 
and the support cycles may not be well-defined.   For a cycle $H = \sum_i h_i Z_i$ on the orbit of some closed subset $Z$, we write $|H|$ for the absolute value $\sum_i |h_i| Z_i = \max(H, -H)$. 
\begin{proposition}
\label{support-prop}
Let $A = \sum I_n t^n \subseteq R[t, t^{-1}; \sigma]$ satisfy Hypothesis~\ref{UFD-hyp}.
Let $p$ be any (not necessarily closed) point of $X$; let $Z$ be the closure of $p$ and 
write $Z_i = \sigma^{-i}(Z)$ for all $Z$.    Let $\{ E_n \}$ be the reduced support cycle sequence along this orbit, 
and assume that $E_n \neq 0$ for some $n$.  

There is a uniquely determined pleasantly alternating cycle $G = \sum a_i Z_i$ such that defining $G_n$ as in Definition~\ref{def:Gn}, one of the following cases holds 
for all $n \in \mb{Z}$: (i) $E_n = |G_n|$; (ii) $E_n = \max(G_n, 0)$; or (iii) $E_n = \max(-G_n, 0)$.
\end{proposition}
\begin{proof}
While our main interest is in the reduced cycle support sequences, it is useful to track multiplicity and consider the support cycle sequence $\{ F_n \}$ along the orbit of $Z$ until the end of the proof.  

We know that $I_0 = R$,  and that $I_m \sigma^m(I_n) \subseteq I_{m+n}$ for all $m,n \in \mb{Z}$.
Since $A$ is quasi-finitely generated, so is $A_{\geq 0}$ by Lemma~\ref{lem:fg}.   Then there is $r > 0$ such that 
$I_n = \sum_{i = 1}^r I_i \sigma^{i}(I_{n-i})$, all $n > r$.  
It now easily follows from the properties of multiplicity 
that $F_0 = 0$, 
$F_m + \sigma^{-m}(F_n) \geq F_{m+n}$ for all $m, n \in \mb{Z}$, and
$F_n = \min_{i=1}^r (F_i + \sigma^{-i}(F_{n-i}))$ for all $n > r$.  This shows that the sequence $\{ -F_n | n \geq 0 \}$
satisfies the definition of $\sigma$-divisor sequence, except that the $-F_n$ are cycles and may not be divisors; but as usual, 
only the combinatorics of the coefficients is relevant.  In particular, Lemma~\ref{comb-lem} holds just as well for cycles.  By that result there are cycles $D, \Omega$  on the orbit of $Z$
with $\Omega$ effective, such that defining $D_n$ in terms of $D$ as in Definition~\ref{def:Gn}, 
we have $-F_n = D_n - \Omega$ for all $n \gg 0$.   Replacing $D$ by $-D$ for convenience we get
 $F_n = D_n + \Omega$ for all $n \gg 0$.

Though $F_n$ and $\Omega$ are effective, it may be that $D_n$ is not effective for all $n \geq 0$, and so we want to adjust 
the form of $F_n$ further.   Set $\Phi = 0$.  Then certainly $F_n = D_n + \Omega + \sigma^{-n}(\Phi)$ for all $n \gg 0$, 
and we claim that we can adjust the choices of $D, \Omega$, and $\Phi$ so that this formula continues to hold for $n \gg 0$, 
$\Omega$ and $\Phi$ remain effective, and $D_n$ is also effective for $n \gg 0$.

By shifting the indexing along the orbit if necessary, we may assume that $D = \sum_{i=0}^m a_i Z_i$.  
Write $D_n = \sum_i a_{n,i} Z_i$.   Defining $d_i = \sum_{j \leq i} a_j$, $e_i = \sum_{j > i} a_j$, and $d = \sum_j a_j$, 
the same proof as in Lemma~\ref{lem:converse} shows that for $n \gg m$ the formula for $D_n$ is
\[
D_n = d_0 Z_0 + d_1 Z_1 + \dots + d_{m-1} Z_{m-1} +
d Z_m + d Z_{m+1} + \dots + d Z_{n-1} + e_{m-1} Z_{n} + e_{m-2} Z_{n+1}
+ \dots + e_0 Z_{n + m -1}.
\]
Now $F_n = D_n + \Omega$ is effective for all $n \gg 0$.   Since for $n \gg 0$ adding $\Omega$ will not counteract a negative $e_i$ 
in the formula for $D_n$, we must already have $e_i \geq 0$ for all $i$, and the problem if any is that some $d_i$ are negative.
We show that we can make an adjustment to $D, \Omega, \Phi$ so that the sum of all $d_i$'s which happen to be negative increases, while the $e_i$'s remain nonnegative.  Then the result follows by induction.  

Let $i$ be minimal such that $d_i < 0$.  Replace $D, \Omega, \Phi$ by 
$D' = D + Z_i - Z_{i+1}$, $\Omega' = \Omega - Z_i$, $\Phi' = \Phi + Z_i$.  It is trivial 
to check that $F_n = D_n + \Omega + \sigma^{-n}(\Phi) = D'_n + \Omega' + \sigma^{-n}(\Phi')$ for $n \gg 0$ 
still holds.   Writing $D' = \sum_i a'_i Z_i$ and $d'_i = \sum_{j \leq i} a'_j$, $e'_i = \sum_{j \geq i} a'_j$,
we see that $d'_k = d_k$ and $e'_k = e_k$ for $k \neq i$, while $d_i' = d_i + 1, e_i' = e_i - 1$.  Since $d_i < 0$ and 
$d_i + e_i = d \geq 0$, $e_i > 0$ and so $e'_i \geq 0$.  Clearly we have improved the negativity of $d_i$, 
while leaving all of the $e_j$ nonnegative.  Also, since we had $d_i < 0$ but $F_n = D_n + \Omega$ was effective for all $n \gg 0$,  
$\Omega$ must have had a positive coefficient for $Z_i$.  Thus $\Omega'$ remains effective and clearly $\Phi'$ is effective since $\Phi$ was.  This completes the induction step, and thus proves the claim.

We now have $F_n = D_n + \Omega + \sigma^{-n}(\Phi)$ for all $n \gg 0$, where $\Omega$ and $\Phi$ are effective, and $D_n$ 
is effective for $n \gg 0$.   Suppose that $\Omega > 0$; we will show that this contradicts the simplicity of $A$.   
For $n \gg m \gg 0$, 
the equation $F_{-m} + \sigma^m(F_n) \geq F_{n-m}$ translates to
\[
F_{-m} + \sigma^m(D_n) + \sigma^m(\Omega) + \sigma^{m-n}(\Phi)
\geq D_{n-m} + \Omega + \sigma^{m-n}(\Phi),
\]
which forces
\[
F_{-m} + \sigma^m(D_m) + \sigma^m(\Omega)
\geq \Omega,
\]
since $\sigma^m(D_n) = D_{n-m} + \sigma^m(D_m)$ by Lemma~\ref{G-lem}(1).

Now note that $I_m \sigma^m(I_{-m}) $ has support divisor $F_m + \sigma^{-m}(F_{-m}) \geq F_m \geq \Omega$ for $m \gg 0$, and 
 $I_{-m}\sigma^{-m}(I_m)$ has support divisor $F_{-m} + \sigma^m(D_m) + \sigma^m(\Omega) + \Phi \geq \Omega$ for $m \gg 0$ by the calculation above.  This shows that for any large $n$, $\sum_{m \geq n} I_m \sigma^m(I_{-m}) + I_{-m} \sigma^{-m}(I_m) \neq R$, 
contradicting Lemma~\ref{simple-crit-lem1}.  Thus $\Omega = 0$.  
A similar proof, which we leave to the reader, shows that $\Phi = 0$.  We conclude that $F_n = D_n$ for all $n \gg 0$.

Aanalogous results hold for the negative degree pieces of $A$.  The quickest way to verify this 
is to use the anti-automorphism $\psi: R[t, t^{-1}; \sigma] \to R[t, t^{-1}; \sigma]$
of Lemma~\ref{lem:anti}, where $\psi(xt^n) = \sigma^{-n}(x) t^{-n}$.   Let $A' = \psi(A)$; then $A'$ also satisfies the hypothesis of this proposition.  Writing 
$A' = \sum_{n \in \mb{Z}} I'_n t^n$ and defining the corresponding support divisors $F'_n$ along the orbit of $p$, note that 
we have $I'_n = \sigma^n(I_{-n})$ and so $F'_n = \sigma^{-n}(F_{-n})$.
By the first part of the proof applied to $A'$, there is a cycle $D'$ 
such that $F'_n = D'_n$ for $n \gg 0$, where $D'_n$ is defined in terms of $D'$ as in Definition~\ref{def:Gn}.
  Then $F_{-n} = \sigma^n(F'_n) = \sigma^n(D'_n) = - D'_{-n}$ for $n \gg 0$, using Lemma~\ref{lem:easyGfacts}.

Now setting $J_n = I_n\sigma^{n}(I_{-n})$, note that Lemma~\ref{simple-crit-lem1} can be restated to 
say that $\sum_{m \geq n} J_m+ \sigma^{-m}(J_m) = R$ for any $n \geq 1$. 
Consider $C = D   + D'$.  The support divisor of the ideal $J_n$ is $F_n + \sigma^{-n}(F_{-n})$, which is 
equal to $D_n+\sigma^{-n}(-D'_{-n}) = D_n + D'_n = C_n$ for $n \gg 0$, where $C_n$ is also 
defined in terms of $C$ as in Definition~\ref{def:Gn}.   Thus we must have 
$\min_{m \geq n} (\min(C_m, \sigma^m(C_m))) = 0$ for $n \gg 0$.
By Lemma~\ref{lem:converse} we get that $C = cG$, for some pleasantly alternating cycle $G$ and some $c \geq 0$.
It is automatic that we also have $\min_{m \geq n} \min(D_m, \sigma^m(D_m)) = 0$ and 
$\min_{m \geq n} (D'_m, \sigma^m(D'_m)) = 0$ for $n \gg 0$, and so by Lemma~\ref{lem:converse} we also get that 
$D$ and $D'$ are nonnegative scalar multiples $dH$ and $d'H'$ of pleasantly alternating cycles $H$ and $H'$, respectively.
Note that $d = d' = 0$ is not possible, for if $F_n = 0$ for all $n \gg 0$ and $n \ll 0$ then $F_m = 0$ for all $m$ because of 
$F_n + \sigma^{-n}(F_{-n+m}) \geq F_{m}$, and this contradicts the hypothesis.

We assume for the rest of the proof that $d > 0$ and $d' > 0$, and we show that case (i) occurs.  If either $d' = 0, d > 0$ 
or $d' > 0, d = 0$, it is easy to adjust the proof below to show that case (ii) or (iii) occurs, respectively, and we leave this to the reader.   
Now since $cG = dH + d'H'$ and $G, H, H'$ are all pleasantly alternating, it is easy to see that this forces $G = H = H'$. 
Also, $G$ is uniquely determined since for any $n \gg 0$, $D_n - \sigma^{-1}(D_{n-1}) = F_n - \sigma^{-1}(F_{n-1})$ is a multiple of $G$.

Consider the reduced support cycle sequence $\{ E_n \}$.  Define $G_n$ in terms of $G$ as in Definition~\ref{def:Gn}, 
and write $E_n = \sum_i e_{n,i} Z_i$ and $G_n = \sum_i g_{n,i} Z_i$.
For $n \gg 0$, $F_n$ is a positive multiple of $G_n$, and since $G_n$ has coefficients in $\{0, 1\}$ by Lemma~\ref{G-lem}(1), 
we have $E_n = G_n$ for $n \gg 0$.  Similarly, $E_n = -G_n$ for $n \ll 0$.   Note also 
that for any $m,n \in \mb{Z}$, the equation $F_m + \sigma^{-m}(F_n) \geq F_{m+n}$ easily implies 
that $E_m + \sigma^{-m}(E_n) \geq E_{m+n}$.

Now fix any $m \in \mb{Z}$.   Picking $n \gg 0$, we have
$E_m \geq E_{m+n} - \sigma^{-m}(E_n) = G_{m+n} - \sigma^{-m}(G_n) = G_m$.    The same argument working 
with $n \ll 0$ shows that $E_m \geq -G_m$.  Thus $E_m \geq \max(G_m, -G_m) = |G_m|$ for all $m$.
Conversely, to show that $E_m \leq |G_m|$ it suffices to prove for all $i$ that if $g_{m,i} = 0$ then $e_{m,i} = 0$.  
Suppose that $g_{m,i} = 0$.  For any $n$, $G_m  =  G_n + \sigma^{-n}(G_{-n+m})$ so that $g_{m,i} = g_{n,i} +  g_{-n+m, i- n}$.
  By Lemma~\ref{G-lem}, we have  $g_{n,i} = 0$ for $n \gg 0$ or for $n \ll 0$.  
In the former case we get $g_{-n+m, i-n} = 0$ for $n \gg 0$ and so 
$E_m \leq E_n + \sigma^{-n}(E_{-n+m})$ implies that  
$e_{m,i} \leq e_{n,i} + e_{-n+m, i-n} = g_{n,i} - g_{-n+m, i-n} = 0$.  
A similar argument applies if $g_{n,i} = 0$ for $n \ll 0$.
Thus $E_m = |G_m|$ for all $m \in \mb{Z}$ as in case (i).
\end{proof}
\begin{definition}
\label{def:equiv}
Assume the hypothesis and notation of Proposition~\ref{support-prop}.  Given any point $p$
such that the support divisor sequence along the orbit of $p$ is nonzero, Proposition~\ref{support-prop} produces 
a unique corresponding pleasantly alternating cycle $G = G(p) = \sum_i g_i Z_i$, where $Z$ is the closure of $p$.  If $q$ is another 
such point with closure $Z'$, leading to a pleasantly alternating cycle $G' = \sum_i g'_i Z'_i$, we write $p \equiv q$ if $g_i = g'_i$ for 
all $i$.    
\end{definition}

Next, we consider the global support of the $R/I_n$ for our ring $A = \bigoplus I_n t^n$.   
\begin{lemma}
\label{Y-lem}
Let $A = \bigoplus_{n \in \mb{Z}} I_n t^n \subseteq R[t, t^{-1}; \sigma]$ satisfy Hypothesis~\ref{UFD-hyp}.
Then there is a closed set $Y \subsetneq X$ with the following properties: 
\begin{enumerate}
\item[(i)] for all $n \in \mb{Z}$, $\spec R/I_n$ is contained in  $\bigcup_{n \in \mb{Z}} \sigma^n(Y)$; 
\item[(ii)] For each connected component $W$ of $Y$, given any two points $p, q \in W$
we have $p \equiv q$; and
 \item[(iii)] $Y$ is a $\sigma$-lonely subset of $X$.
\end{enumerate}
\end{lemma}
\begin{proof}
We have $I_n = \sum_{i = 1}^r I_i \sigma^{i}(I_{n-i})$ for all $n > r$, some $r \geq 1$, since $A_{\geq 0}$ is quasi-finitely generated by Lemma~\ref{lem:fg}.   Similarly, $A_{\leq 0}$ is quasi-finitely generated, for example by applying Lemma~\ref{lem:fg} to the image 
of $A$ under the anti-isomorphism $\psi$ of Lemma~\ref{lem:anti}.  Thus 
$I_{-n} = \sum_{i = 1}^r I_{-i} \sigma^{-i}(I_{-n + i})$ for all $n > r$ also (choosing a common $r$ that works for 
both the positive and negative degree parts).  Now let $Y$ be the union of all of the closed subsets $\spec R/I_i$ for $-r \leq i \leq r$.    It is then easy to prove by induction from the equations above that $\bigcup_{n \in \mb{Z}} \sigma^n(Y)$ contains $\spec R/I_n$ for all $n$, so that (i) is satisfied.

Suppose that $p$ and $p'$ are points in $Y$ with respective closures $Z$ and $Z'$, and assume that $p'$ is in the closure of $p$, 
in other words that $Z' \subseteq Z$.  Let $\{ E_n \}$ and $\{E'_n \}$ 
be the reduced support divisor sequences of the orbits of $p$ and $p'$, respectively.  Applying Proposition~\ref{support-prop} produces respective pleasantly alternating cycles $G = G(p)$ and $G' = G(p')$.  Write $G_n = \sum g_{n,i} Z_i$ and $G'_n = \sum g'_{n,i} Z'_i$.  Because we want to compare divisors on orbits of the same closed set, we transfer $G_n$ to the orbit of $Z'$ to get $\wt{G}_n = \sum g_{n,i} Z'_i$. 
Suppose that case (i) or (ii) holds for the orbit of $p$; the proof in case (iii) is similar and left to the reader.   Then $E_n = G_n$ for all $n \gg 0$, where $G_n$ is defined in terms of $G$ as in Definition~\ref{def:Gn}.    Since $m_p(I) > 0$ implies $m_{p'}(I) > 0$ for any ideal $I$ of $R$, we must have case (i) or (ii) for the orbit of $p'$ as well, so $E'_n = G'_n$ for $n \gg 0$ and $\wt{G}_n \leq G'_n$ for all $n \gg 0$. 
Now note that any pleasantly alternating cycle $G$ satisfies $\deg G_n = n$ for all $n$, where degree means the sum of the coefficients.  This forces $\wt{G}_n = G'_n$ for all $n \gg 0$, and finally this 
implies that $\wt{G} = \wt{G}_n - \sigma^{-1}(\wt{G}_{n-1}) = G'_n - \sigma^{-1}(G'_{n-1}) = G'$.  In other words, 
$p \equiv p'$ in the notation of Definition~\ref{def:equiv}.

In particular, if the closure $Z$ of $p$ is an irreducible component of $Y$, then all points $p' \in Z$ 
have $p' \equiv p$.  Then if $W$ is a connected component of $Y$, this implies that 
$p \equiv q$ for all $p, q \in W$.   We now claim that each connected component $W$ of $Y$ is $\sigma$-lonely.  For, suppose that $p \in W \cap \sigma^j(W)$ for some $j$.   Then $p = \sigma^j(q)$ for some $q \in W$, so that $p = \sigma^j(q) \equiv q$.   But this is clearly impossible
unless $j = 0$, so $W$ is $\sigma$-lonely as claimed.

Finally, suppose that $W_1$ and $W_2$ are two distinct connected components of $Y$ and that 
$W_1 \cap \sigma^j(W_2) \neq \emptyset$ for some $j \in \mb{Z}$.  We may replace $W_2$ with $\sigma^j(W_2)$, 
obtaining a new $Y$ which still satisfies condition (i), but now has fewer connected components.   Continuing this process if necessary, we arrive at a $Y$ whose decomposition into connected components $Y  = W_1 \cup \dots \cup W_m$ has the property that each $W_i$ is $\sigma$-lonely, and $W_i \cap \sigma^k(W_j) = \emptyset$ for all $k \in \mb{Z}$ 
if $i \neq j$.  Then $Y$ is itself $\sigma$-lonely, proving (iii).  That condition (ii) holds for each connected component $W_i$ of $Y$ was shown earlier in the proof.
\end{proof}

We are now ready to prove our classification theorem.

\begin{theorem}
\label{thm:class}
Let $A' = \bigoplus_n W_n t^n \subseteq K[t, t^{-1}; \sigma]$ satisfy  Hypothesis~\ref{UFD-hyp}.   
Choose an invertible $R$-module $M$ such that the corresponding $\Pic(X)$-twist $A$ satisfies 
$A = \bigoplus_n M_n W_n t^n \subseteq R[t, t^{-1}; \sigma]$, as in Lemma~\ref{lem:trivchange}.

Then there is a $\sigma$-lonely subset $Z \subseteq X$, which is a disjoint union of 
connected components $Z^{(1)}, \dots, Z^{(m)}$, and pleasantly alternating cycles $G^{(i)}$ on the orbit of $Z^{(i)}$, together with ideals $H^{(i)}$, $J^{(i)}$ of $R$ with $R/H^{(i)}$ and $R/J^{(i)}$ supported on $Z^{(i)}$, such that $A = \bigcap_{i = 1}^m B(G^{(i)}, H^{(i)}, J^{(i)})$.  Moreover, $A$ is graded Morita equivalent to 
$B(Z, H, J)$, where $H = \bigcap_i H^{(i)}$ and $J = \bigcap_i J^{(i)}$.
\end{theorem}
\begin{proof}
The passage from $A'$ to $A$ is described by Lemma~\ref{lem:trivchange}.  We claim that 
$A$ also satisfies Hypothesis~\ref{UFD-hyp}.   Lemma~\ref{lem:trivchange} shows that $A$ is still simple, and it is easy to see that $A$ remains quasi-finitely generated.
To see that $A$ is an Ore domain, note that $A$ is obviously still a graded Ore domain, with 
graded ring of fractions $Q = Q_{\rm gr}(A) \cong K[t, t^{-1}; \sigma]$ still.  But this is implies that $A$ is an Ore domain, since $A$ has a localization, namely $Q$, which is noetherian and hence Ore.

Write $A = \bigoplus_n I_n t^n$ with $I_n \subseteq R$ 
and let $Z$ be the $\sigma$-lonely subset $Y$ defined by Lemma~\ref{Y-lem}.  We let $Z^{(1)}, \dots, Z^{(m)}$ be the connected 
components of $Z$.
For any ideal $I$ such that $\spec (R/I) \subseteq \bigcup_{k \in \mb{Z}} \sigma^k(Z)$ as sets, 
we may write uniquely $I = I^{(1)} \cap \dots \cap I^{(m)}$
where $R/I^{(i)}$ is supported on $\bigcup_k \sigma^k(Z^{(i)})$.
Then for each $i$ we define $A^{(i)} = \bigoplus_n I_n^{(i)} t^n$.  It is obvious that each $A^{(i)}$ is 
again a subalgebra of $R[t, t^{-1}; \sigma]$, and that 
$A = A^{(1)} \cap \dots \cap A^{(m)}$.  Suppose that the result is true in the special case that $m = 1$ and so 
$Z$ is connected.  Then for each $i$ we have $A^{(i)} \cong B(G^{(i)}, H^{(i)}, J^{(i)})$ for some data $G^{(i)}$, $H^{(i)}$, $J^{(i)}$.
By Proposition~\ref{prop:Morita} and its proof, there is a  $B^{(i)} = B(Z^{(i)}, H^{(i)}, J^{(i)})$-module $L^{(i)} \subseteq R[t, t^{-1}; \sigma]$ such that $\End_{B^{(i)}}(L^{(i)}) = A^{(i)}$, and this gives a graded Morita equivalence between $A^{(i)}$ and $B^{(i)}$.
Then $B = B(Z, H, J) = B^{(1)} \cap B^{(2)} \cap \dots \cap B^{(m)}$ and an easy argument shows that $L = \bigcap L^{(i)}$ is a $B$-module such that $\End_B(L) \cong A$, and $L$ gives a graded Morita equivalence between $A$ and $B$.

Thus we may reduce to the case that $m = 1$ and hence $Z = Z^{(1)}$ is connected, for the rest of the proof.  
It remains to find $G, H, J$ so that $A \cong B(G, H, J)$.  By Lemma~\ref{Y-lem}(2), given any point $p \in Z$ 
with closure $V$, the plesantly alternating cycle $G(p) = \sum g_i V_i$ arising from Proposition~\ref{support-prop}
has coefficients independent of the choice of $p$.   We now write the formal cycle $G = \sum g_i Z_i$, and let 
$G_n = \sum_i g_{n,i} Z_i$ be defined as in Definition~\ref{def:Gn} for each $n$.  
We can write uniquely $I_n = \bigcap_i I_{n,i}$, where 
$\spec R/I_{n,i} \subseteq Z_i = \sigma^{-i}(Z)$.    Note that if $g_{n,i} = 0$, then $I_{n,i} = R$.  This is clear from the result of Proposition~\ref{support-prop} and the fact that $p \equiv q$ for all $p, q \in Z$.

By Lemma~\ref{G-lem} there is $N \geq 1$ such that that for $n \geq N$, $G_n$ and $-G_{-n}$ are effective.
The remainder of the proof is to show that there is are ideals $H,J$ 
such that $I_{n,i} = \sigma^{i}(J)$ if  $g_{n,i} = 1$ and $I_{n, i} = \sigma^{i}(H)$ if $g_{n,i} = -1$.  Then we will have $A = B(G, H, J)$ by definition.

Fix any $e \geq N$, where $N$ is as above.
We have $G_{m+e} = G_m + \sigma^{-m}(G_e)$, where the coefficients of $G_e$ are in $\{0, 1 \}$ by Lemma~\ref{G-lem}.
Suppose that $g_{m,i} = 1$ for some $m$.  Then $g_{m+e, i}  = g_{m,i} + g_{e, i-m}$
forces $g_{e, i-m} = 0$ and $g_{m+e, i} = 1$, since $G_{m+e}$ cannot have a coefficient of $2$.
Using that $G_m = - \sigma^{-m}(G_{-m})$, 
we have $g_{n,i} = - g_{-n,i-n}$ and so 
$g_{e, i-m} = 0$ implies $g_{-e, i-m-e} = 0$ also.
By what we have already shown above, $g_{e, i-m} = 0$ implies $I_{e, i-m} = R$ and
$g_{-e, i-m-e} = 0$ implies $I_{-e,i-m-e} = R$.
The equation $I_{m+e} \supseteq I_m \sigma^m(I_e)$ implies $I_{m+e, i} \supseteq I_{m,i} \sigma^m(I_{e, i-m}) = I_{m,i}$.  Similarly, $I_m \supseteq I_{m+e} \sigma^{m+e}(I_{-e})$ 
implies that $I_{m,i} \supseteq I_{m+e,i}\sigma^{m+e}(I_{-e, i-m-e}) = I_{m+e, i}$.  Thus $I_{m,i} = I_{m+e,i}$.
In summary, we have shown that if $g_{m,i} = 1$ and $e \geq N$, then $g_{m+e, i} = 1$ also and $I_{m,i} = I_{m+e,i}$.  

We claim that if $g_{m,i} = 1$ then we also have $g_{m+e, i+e} = 1$ and 
$I_{m+e, i+e} = \sigma^e(I_{m,i})$.  This is a very similar argument.
Using $G_{m+e} = G_e + \sigma^{-e}(G_m)$ we get $g_{m+e, e+i} = g_{e, e+i} + g_{m,i}$ and hence 
$g_{e, e+i} = 0$ and $g_{m+e, e+i} = 1$.   Then $g_{-e, i} = -g_{e, e+i} = 0$, so $I_{e, e+i} = R = I_{-e, i}$.  
The equations 
$I_{m+e} \supseteq I_{e} \sigma^e(I_m) $ and $I_{m} \supseteq I_{-e} \sigma^{-e}(I_{m+e})$ 
imply that $I_{m+e, e+i} \supseteq \sigma^e(I_{m,i})$ and $I_{m,i} \supseteq \sigma^{-e}(I_{m+e, i+e})$, 
which together give the claim.

Now without loss of generality we may shift indices so that $G = \sum_{i = 0}^m a_i Z_i$, where $a_0 = 1$.  Then 
$g_{1, 0} = 1$ and we let $J = I_{1, 0}$.  Suppose that $g_{m,i} = 1$, and choose $a,b, c,e \gg N$ such that $i + e = b$ and $m + a + e = 1 + b + c$.   
 Then using the isomorphisms above we have 
\[
I_{m, i} = I_{m+a, i} = \sigma^{-e}(I_{m+a + e, i+e}) = \sigma^{-e}(I_{1+b+c, b})
\]
and
\[
\sigma^i(J)  = \sigma^i(I_{1,0}) =  \sigma^{i}(I_{1+c,0}) = \sigma^{-e+b}(I_{1+c, 0}) = \sigma^{-e}(I_{1+b+c, b}).
\]
In other words, we have shown that if $g_{m,i} = 1$ then $I_{m,i} = \sigma^i(J)$, as we wished.

To construct the ideal $H$, we may use the anti-automorphism $\psi: xt^n \mapsto \sigma^{-n}(x) t^n$ from 
Lemma~\ref{lem:anti}.  Let $\wt{A} = \psi(A)$.
Write $\wt{A} = \bigoplus_n \wt{I}_n t^n \subseteq R[t, t^{-1}; \sigma]$, and apply the proof already given to $\wt{A}$. 
Writing $\wt{I}_n = \bigcap_i \wt{I}_{n,i}$ with $\spec R/\wt{I}_{n,i} \subseteq Z_i$, then 
$\wt{I}_{n,i} = \sigma^n(I_{-n,i-n})$.  It is easy to check that the associated pleasantly alternating cycle  is actually the same as $G$.  
Setting $H = \wt{I}_{1,0}$ now, we get if $g_{m,i} = 1$ then $\wt{I}_{m,i} = \sigma^i(H)$ by what has already been proved.
Now if $g_{m,i} = -1$, 
then $g_{-m, i-m} = 1$ and $I_{m,i} = \sigma^m(\wt{I}_{-m, i-m}) = \sigma^m(\sigma^{i-m}(H)) = \sigma^i(H)$.  
This finishes the proof.
\end{proof}

\section{$\mb{Z}$-graded birationally commutative simple domains of minimal GK-dimension}
\label{sec:minGK}

In this section, we explore the special case of the above results for birationally commutative $\mb{Z}$-graded simple 
rings $A$ where $\GK A$ is as small as possible.  We will see that this restriction on the GK-dimension 
further constrains the structure of the rings in a significant way.   We remind the reader of our global hypothesis that the base field $k$ is algebraically closed, which will be important in several proofs in this section.

Suppose that $A$ is a $\mb{Z}$-graded $k$-algebra which is an Ore domain with graded ring of fractions $K[t, t^{-1}; \sigma]$, where $K$ is 
a field with $\trdeg(K/k)  = d < \infty$.  Then it easy to show that $\GK (A) \geq d + 1$, by considering a subalgebra of 
the form $k \langle V + ka \rangle$, where $V \subseteq K$ is the span of a transcendence basis for $K/k$ and $a$ is any nonzero element of positive degree.   Example~\ref{ex:weyl-loc} shows that the GK-dimension of $A$ is bigger than $d + 1$ in general.   The next result, which is a variation of an idea of James Zhang from \cite{Zh}, 
shows that $\GKdim(A) = d+1$ can occur only for $K$ and $\sigma$ with a special property.

\begin{lemma}
\label{fixed-V-lem} Let $A$ be a finitely generated $\mb{Z}$-graded $k$-algebra with $A_n \neq 0$ for all $n \in \mb{Z}$.  
Assume that $A$ is a birationally 
commutative Ore domain with $Q_{\operatorname{gr}}(A) = K[t, t^{-1}; \sigma]$.
Suppose that $\GK A = \trdeg K/k + 1$.

Then for every finite-dimensional $k$-subspace $W \subseteq A_0$, there is a
finite-dimensional vector $k$-subspace $V \subseteq K$ with $W
\subseteq V$ and  $\sigma(V) = V$.
\end{lemma}
\begin{proof}
Since $A_1 \neq 0$, we can change the choice of $t$ (by replacing $t$ by $zt$ for some $z \in K$) so that $t \in A_1$.  Note that $K/k$ is a finitely generated field extension, and $K$ is the fraction field of $A_0$, by Lemma~\ref{D-fg-lem}.   Let $\trdeg K/k = d$.

In \cite{Zh}, Zhang defines the following concept.  Given a
commutative $k$-algebra $C$ of GK-dimension $d$, the algebra $C$ satisfies the
\emph{sensitive multiplicity condition} $SM(U_0, c, d)$ if there is
a finite dimensional $k$-subspace $U_0 \subseteq C$ and a constant
$c > 0$ with the following preoprty:  for every finite-dimensional $k$-subspace $W
\subseteq C$ with $U_0a \subseteq W$ for some regular element $a \in
C$, one has $\dim_k W^n \geq c (\dim_k W) n^d$ for all $n \geq 0$.
Now \cite[Theorem 3.2]{Zh} shows that since $K/k$ is a finitely
generated field extension of transcendence degree $d$ and $k$ is algebraically closed, then there
is a $U_0 \subseteq K$  and $c > 0$ such that $K$ satisfies $SM(U_0,
c, d)$.

Now let $W$ be any finite-dimensional $k$-subspace of $A_0$.  We need to find
a $\sigma$-stable finite-dimensional subspace $V \subseteq K$ which contains $W$.
Clearly it does no harm to enlarge $W$, so we assume that $1 \in W$
and since $K$ is the fraction field of $A_0$, we may assume that $W$ generates $K$ as a field.
Writing a basis of $U_0$ as fractions of elements in $A_0$ with a common denominator, 
we see that there is $0 \neq a \in A_0$ such that $U_0a \subseteq A_0$.
Enlarging $W$ further to $W + U_0a$, we can assume that $U_0a \subseteq W$.

The rest of the proof closely follows the same idea as the proof of \cite[Theorem 3.2]{Zh}, but we reproduce 
the argument here for the convenience of the reader.  Consider the subalgebra $B = k \langle W + kt \rangle \subseteq
A$.   A direct calculation shows that 
\[
(W + kt)^n = \bigoplus_{i = 0}^n (W + \sigma(W) + \dots + \sigma^i(W))^{n-i} t^i.
\]
Write $W_i = W + \sigma(W) + \dots + \sigma^i(W)$, and suppose that $\dim_k W_i$ is an unbounded function of $i$.
We have $\dim_k W_i \geq i +1$, for all $i$, since if $\dim_k W_i = \dim_k W_{i+1}$, say, then 
$W_n = W_i$ for all $n \geq i$, a contradiction.
Now since $K$ satisfies $SM(U_0, c, d)$ and $U_0 a \subseteq W$, we have $\dim_k W_i^{n-i} \geq c (\dim_k W_i) (n-i)^d$.
Then 
\[
\dim_k (W + kt)^n \geq \sum_{i = 0}^n c (\dim_k W_i) (n-i)^d \geq \sum_{i = 0}^n ci(n-i)^d \geq C n^{d+2}
\]
for some constant $C$, and so $\GK A \geq \GK B \geq d + 2$, contradicting the hypothesis.  
 Thus $\dim_k W_i$ is a bounded function of $n$, and so some $W_i \supseteq W$ satisfies $\sigma(W_i) = W_i$.
\end{proof}

The previous result applies even to commutative rings $A$, but in this paper we are interested primarily in 
$\mb{Z}$-graded rings that are highly noncommutative in some sense, in particular simple. 
We say that an algebra $B$ over a field $k$ is \emph{centerless} if its center $Z(B)$ is equal to $k$.  
The ring $Q = K[t, t^{-1}; \sigma]$, where $K$ is a field, has center $Z(Q) = \{a \in K | \sigma(a) = a \}$, as can easily be checked.
In most common circumstances, the graded quotient ring $Q$ of a simple birationally commutative $\mb{Z}$-graded algebra $A$ 
will be automatically centerless (see Theorem~\ref{thm:inT}(3) below).   Thus, we now study the further restrictions 
on $(K, \sigma)$ that arise from a centerless hypothesis on $Q$.

\begin{proposition}
\label{rat-prop}
Let $\sigma: K \to K$ be an automorphism of a
finitely generated field extension $K/k$. Assume that $Q = K[t, t^{-1}; \sigma]$ is centerless, in other words that $\{ a \in K | \sigma(a) = a
\} = k$.
\begin{enumerate}
\item
Let $T$ be the sum of all generalized eigenspaces of $\sigma$ inside $K$.
Then $T$ is a $\sigma$-invariant subring of $K$ and there are $x_1, \dots, x_m \in K$, algebraically independent over $k$, such that
either 
\begin{enumerate}
\item[(A)] $T = k[x_1, x_2^{\pm 1}, \dots, x_m^{\pm 1}]$, $\sigma(x_1) = x_1 + 1$, and $\sigma(x_i) = p_ix_i$ for some $p_i \in k$, $i \geq 2$; or 
\item[(B)]$T = k[x_1^{\pm 1}, x_2^{\pm 1}, \dots, x_m^{\pm 1}]$ and $\sigma(x_i) = p_ix_i$ for some $p_i \in k$, $i \geq 1$.
\end{enumerate}
In each case the $p_i$ generate a free abelian subgroup of $k^{\times}$, and if Case (A) occurs then $\cha k = 0$.

\item If there exists a finite-dimensional $k$-subspace $V \subseteq K$
such that $V$ generates $K$ as a field and $\sigma(V) = V$, then $T$ has fraction field $K$
and $K \cong k(x_1, \dots, x_m)$ with automorphism as in case (A) or (B) above.
\end{enumerate}
\end{proposition}
\begin{proof}
(1)  For $\lambda \in k$, let $V_{\lambda} \subseteq K$ be the generalized eigenspace for the eigenvalue $\lambda$, 
in other words $V_{\lambda} = \{ z \in K | (\sigma - \lambda)^n z = 0\ \text{for some}\ n \geq 1 \}$.
 Let $\Lambda \subseteq k$ be the set of all eigenvalues of the action of $\sigma$ on $K$, and let $T = \bigoplus_{\lambda \in \Lambda} V_{\lambda} \subseteq K$ (note that the sum is automatically direct).  Inside $V_{\lambda}$ we have 
the $\lambda$-eigenspace $W_{\lambda} = \{z \in K | \sigma(z) = \lambda z \}$, and the centerless hypothesis on $Q$ is 
equivalent to $W_1 = k$.  If $0 \neq v, w \in W_{\lambda}$, then $vw^{-1} \in W_1$ and so $v \in kw$.  Thus each $W_{\lambda}$ has 
dimension $1$, say $W_{\lambda} = k v_{\lambda}$ for some eigenvector $v_{\lambda}$.  
It is easy to check that for any polynomial $f \in k[x]$ and $z \in K$,  
we have $f(\sigma) (zv_{\lambda}) = f(\lambda \sigma) (z)$.  In particular, applying this to $f = (x -\lambda)^n$ for all $n$ shows that $V_{\lambda} = v_{\lambda} V_1$.

Suppose that $W_1 \subsetneq V_1$; then we say we are in case (A).   In this case there is $u \in K$ such that $\sigma(u) = u + 1$.   If $k$ has positive characteristic $p$, then $w = u(u+1) \dots (u + p -1)$ is $\sigma$-fixed, so $w \in k$.  This shows that $u$ is algebraic over $k$ and since $k$ is algebraically closed, $u \in k$, a contradiction.  Thus $k$ has characteristic $0$.   

Since $W_1 = k$, every finite-dimensional 
$\sigma$-stable subspace $Y$ of $V_1$ is a single Jordan block for the action of $\sigma$.   Given two such subspaces $Y, Y'$ of the 
same dimension, then $Y + Y'$ is also $\sigma$-stable and acted on as a single Jordan block; then it has a unique flag of $\sigma$-stable 
subspaces and thus $Y = Y'$.  Since $k[u]_{\leq n}$ is $\sigma$-stable of dimension $n + 1$ and contained in $V_1$, it must be the unique such 
subspace of that dimension, and since $V_1$ is the union of its finite-dimensional $\sigma$-fixed subspaces we have $V_1 = k[u]$.  Then  $T = \bigoplus_{\lambda \in \Lambda} v_{\lambda} k [u]$. 
The alternative is case (B), that is that $V_1 = k$.  Then $T = \bigoplus_{\lambda \in \Lambda}  k v_{\lambda}$.  It is clear that in either case $T$ is a ring, since $v_{\lambda} v_{\mu} \in k v_{\lambda \mu}$.

Suppose case (A).  The rest of the proof in case (B) is similar to the proof of case (A) but easier, and is left to the reader.
Suppose that $\Lambda'$ is a finitely generated subgroup of $\Lambda$.  Note that $\Lambda$ is a torsionfree subgroup of $k^{\times}$:  if $\lambda \in \Lambda$ has $\lambda^d = 1$, then $v_{\lambda}^d$ is $\sigma$-fixed and so $v_{\lambda}^d \in k$.  Since $k$ is algebraically closed,
then $v_{\lambda} \in k$ and $\lambda = 1$.  So $\Lambda'$ is a free group of finite rank, say with basis $p_2, \dots, p_m$.
Set $x_i = v_{p_i}$.  Since the $p_i$ are a free basis of $\Lambda'$, distinct words in the $x_i$ are associated to distinct eigenvalues and so are $k$-independent.  Thus $k[x_2, \dots, x_m]$ is a polynomial ring.  It follows setting $u = x_1$ that 
$R(\Lambda') = \bigoplus_{\lambda \in \Lambda'} v_{\lambda} k[u]$ is a ring isomorphic to the 
polynomial ring $k[x_1, x_2, \dots, x_m]$.  Let $K(\Lambda') \cong k(x_1, \dots, x_m)$ be its fraction field.

Suppose that $\Lambda' \subsetneq \Lambda''$ are two distinct finitely generated subgroups of $\Lambda$.
We claim that the fields $K(\Lambda') \subsetneq K(\Lambda'')$ are also distinct.   Choose some $\mu \in \Lambda'' \setminus \Lambda'$.  
If $v_{\mu} \in K(\Lambda')$, then $v_{\mu}= f/g$ where $f, g \in R(\Lambda') = \sum_{\lambda \in \Lambda'} k[u] v_{\lambda}$.
Looking at $f = g v_{\mu}$ in the bigger ring $R(\Lambda'')$, we see that 
$\sum_{\lambda \in \Lambda'} k[u] v_{\lambda} \cap \sum_{\lambda \in \Lambda'} k[u] v_{\lambda \mu} \neq \emptyset$,
contradicting that the cosets $\Lambda'$ and $\Lambda' \mu$ are disjoint.  This proves the claim.

Now since $K/k$ is a
finitely generated field extension, $K$ has ACC on subfields containing $k$.  This implies that $\Lambda$ is a finitely generated group.
Take $\Lambda = \Lambda'$ above.  Then it is clear that $T = \bigoplus_{\lambda} v_{\lambda} k[u] = k[x_1, x_2^{\pm 1}, \dots, x_m^{\pm 1}]$
and that $\sigma$ has the stated form.

(2).  Since $\sigma(V) = V$, $V$ is a sum of generalized eigenspaces.  Thus $V \subseteq T$, and the result follows from part (1).
\end{proof}

The rings $T$ and automorphisms $\sigma$ occurring in the previous result give important examples of wild automorphisms of 
affine varieties.  
\begin{lemma}
\label{T-lem}
Let $(T, \sigma)$ be as in either case (A) or (B) of Proposition~\ref{rat-prop}, where $T$ has fraction field $K = k(x_1, \dots, x_m)$, 
and let $X = \spec T$.
\begin{enumerate}
\item $T$ is  $\sigma$-simple, or equivalently $\sigma: X \to X$ is wild.
\item  If $C \subseteq T$ is a subalgebra with $\sigma(C) = C$, and $C$
is also $\sigma$-simple with field of fractions $K$, then $C = T$.
\end{enumerate}
\end{lemma}
\begin{proof}
We assume case (A) throughout the proof, since the proof for case (B) is similar but easier.
Recall as in the proof of  Proposition~\ref{rat-prop} that $T$ decomposes as a 
sum of generalized eigenspaces of the form $T = \bigoplus_{\lambda \in \Lambda} k[x_1] v_{\lambda}$,
where $\Lambda$ is the free abelian group generated by the $p_i$.   As we also saw in that proof, 
the nonzero $\sigma$-invariant finite-dimensional subspaces of $k[x_1]$ are precisely the $k[x_1]_{\leq n}$ for $n \in \mb{N}$, 
and so the nonzero $\sigma$-invariant subspaces of $k[x_1]$ are the $k[x_1]_{\leq n}$ for $n \in \mb{N} \cup \{ \infty \}$.
Now it is easy to check that any $\sigma$-invariant $k$-subspace $S$ of $T$ is of the form 
$S = \bigoplus_{\lambda \in Y} k[x_1]_{\leq \delta(\lambda)} v_{\lambda}$ for some subset $Y$ of $\Lambda$ and 
constants $\delta(\lambda) \in \mb{N} \cup \{ \infty \}$.

(1).  Suppose that $I$ is a nonzero $\sigma$-invariant ideal of $T$.  As above, write
$I = \bigoplus_{\lambda \in Y} k[x_1]_{\leq \delta(\lambda)} v_{\lambda}$.  Since $I$ is closed under multiplication by $k[x_1]$,
$\delta(\lambda) = \infty$ for all $\lambda \in Y$.  Since $I$ is closed under multiplication by each $v_{\lambda}$
with $\lambda \in \Lambda$, $Y$ is closed under multiplication with anything in $\Lambda$ and so $Y = \Lambda$.  Thus $I = T$.

(2).  Again we write $C = \bigoplus_{\lambda \in Y} k[x_1]_{\leq \delta(\lambda)} v_{\lambda}$.  Since $C$ is a $k$-algebra,
$Y$ is closed under multiplication and $1 \in Y$, so $Y$ is a sub-semigroup of $\Lambda$.  The closure of $C$ under
multiplication also forces $\delta(\lambda_1 \lambda_2) \geq \delta(\lambda_1) + \delta(\lambda_2)$ (with the obvious convention when one or more elements is $\infty$), for all $\lambda_1, \lambda_2 \in Y$.
Suppose that $Y$ is not a group.  Then it has a proper semigroup ideal $0 \subsetneq M \subsetneq Y$ (i.e. $MY \subseteq M$)
and so $\bigoplus_{\lambda \in M} k[x_1]_{\leq \delta(\lambda)} v_{\lambda}$ is a proper $\sigma$-invariant ideal of $C$.
Since $C$ is $\sigma$-simple, this is a contradiction, and so $Y$ is a subgroup of $\Lambda$.  Now the same argument as in the proof of 
Proposition~\ref{rat-prop}(2) shows that if $Y$ is a proper subgroup of $\Lambda$, then the fraction field of $C$ is smaller 
than $K$.  Thus $Y = \Lambda$.

Now consider $\delta$.  If $\delta(\lambda) = 0$ for all $\lambda \in \Lambda$, then
$C \subseteq k[x_2, \dots, x_n]$ does not have the correct field of fractions.  Thus $\delta(\lambda) > 0$ for some $\lambda$.
Then $\delta(1) \geq \delta(\lambda) + \delta(\lambda^{-1}) > 0$ and $\delta(1) \geq \delta(1) +
\delta(1)$, forcing $\delta(1) = \infty$.  This implies that $\delta(\mu) \geq
\delta(\mu) + \delta(1) = \infty$ for all $\mu \in \Lambda$.  So $C = \bigoplus_{\lambda \in \Lambda} k[x_1] v_{\lambda}$.
Thus $C = T$.
\end{proof}

\begin{remark}
\label{rem:wild}
Let $(T, \sigma)$ be as in case (A) or (B) of Proposition~\ref{rat-prop}.  Then $X = \spec T$ is an algebraic group, namely 
$X \cong k \times (k^*)^{m-1}$ in case (A) and in $X \cong (k^*)^m$ in case (B), where
here $k$ is the additive group of the field and $k^*$ is the multiplicative group.  Moreover, the induced
automorphism $\sigma: X \to X$ is a translation automorphism in this algebraic group, and we have 
seen that it is a wild automorphism. 

In \cite{RRZ}, wild automorphisms of projective varieties were studied and it was proved that
in dimension $\leq 2$, all such are translation automorphisms of abelian varieties \cite[Theorem 6.5]{RRZ}.  It was also 
conjectured that the same holds in all dimensions \cite[Conjecture 0.3]{RRZ}.   The examples above suggest that it is also interesting 
to consider the affine version of the wild automorphism problem.  Namely, must a wild automorphism of an affine 
variety be a translation automorphism of a commutative affine algebraic group?  
\end{remark}

We now put together the various pieces above to prove the following summary result.
\begin{theorem}
\label{thm:inT}
Let $A$ be a $\mb{Z}$-graded finitely generated $k$-algebra with $A_n \neq 0$ for all $n \in \mb{Z}$.  
Assume that $A$ is a  birationally commutative Ore domain  
with $Q = Q_{\operatorname{gr}}(A) = K[t, t^{-1}; \sigma]$, and that $\GK(A) = \trdeg(K/k) + 1$.  
\begin{enumerate}
\item  Suppose that $Q$ is centerless.  Then $K = k(x_1, \dots, x_m)$ is a rational function field, 
$\sigma: K \to K$ is of the form as in type (A) or (B) in Proposition~\ref{rat-prop}, and $A_0 \subseteq T$
for the corresponding ring $T$.

\item If $Q$ is centerless and $A$ is simple, then $A_0 = T$.

\item If $A$ is primitive and noetherian and the base field $k$ is uncountable, then $Q$ is automatically centerless.
\end{enumerate}
\end{theorem}
\begin{proof}
(1).   Note that $K$ is the fraction field of $A_0$, and $K/k$ is a finitely generated field extension, by Lemma~\ref{D-fg-lem}.  By Lemma~\ref{fixed-V-lem}, given any finite-dimensional subspace $V$
of $A_0$, then $V$ is contained in a finite-dimensional
subspace $W$ of $K$ with $\sigma(W) = W$.  In particular, taking $V$ to be a subset of $A_0$ which generates $K$ as a field,
there is a finite-dimensional $\sigma$-invariant subspace $W$ of $K$ which generates $K$ as a field.  By  Proposition~\ref{rat-prop}, we get that $K = k(x_1, \dots, x_m)$ is a rational function field, and that $\sigma$ is of type (A) or (B).  The corresponding ring $T$ consists of those elements in $K$ which are sums of generalized eigenvectors for $\sigma$.
Any finite-dimensional $\sigma$-invariant $k$-subspace of $K$ is then contained in $T$, and so we conclude that $A_0 \subseteq T$.

(2).   This is a similar argument as in Lemma~\ref{lem:basics}, but we do not assume that $A_0$ is noetherian so a slightly different 
proof is needed.  Let $C$ be the $k$-subalgebra of $K$ generated by $\{ \sigma^n(A_0) | n \in \mb{Z}\}$.  
The same proof as in Lemma~\ref{lem:basics}(1) shows that $C$ is $\sigma$-simple.  
 Since $A_0 \subseteq T$ by part (1),  $C \subseteq T$ also.   Then $C = T$ by Lemma~\ref{T-lem}(2). 
In particular, $C$ is a finitely generated $k$-algebra.  

Now since $C$ is finitely generated, $C$ is generated by $\{ \sigma^n(A_0) | -r \leq n \leq r \}$ 
for some $r$.  The argument in 
Lemma~\ref{lem:basics}(2) constructs for each $i$ an element $x_i$ such that $x_i \sigma^i(A_0) \subseteq A_0$, 
and thus there is $x = \prod x_i$ such that $x C \subseteq A_0$.  By clearing denominators we can assume that $x \in A_0$.
Now writing $A' = \bigoplus_{n \in \mb{Z}} C A_n$, then $A' \subseteq K[t, t^{-1}; \sigma]$ is a subring 
such that $xA' \subseteq A$.  As in the proof of Lemma~\ref{lem:basics}(2), this forces $A = A'$ since $A$ is simple, and so $A_0 = C = T$.

(3)  This is a standard result.   Namely, \cite[Lemma II.7.13, Proposition II.7.16]{BG} show that under these hypotheses, the center 
of the full quotient ring $Q(A)$ of $A$ is algebraic over $k$, and thus equal to $k$ by our standing hypothesis that $k$ is algebraically closed.
But the center of the graded quotient ring $Q = Q_{\rm gr}(A)$ is no bigger than the center of $Q(A)$. 
\end{proof}

\begin{remark}
Although simple rings are our main interest in this paper, parts (1) and (3) of the theorem above show 
that primitive birationally commutative $\mb{Z}$-graded algebras $A$ of minimal GK-dimension already 
have the same restrictions on their graded quotient ring (but may have more freedom as to what $A_0$ can be).  It would be interesting to try to classify maximal orders, for example, in this more general context.   
\end{remark}

We end this section by stating the special case of our main theorems in this paper for rings of GK-dimension 2, 
and thus proving Theorem~\ref{thm:GK2-intro}.  As mentioned in the introduction, understanding 
this case was the authors' original main motivation.
\begin{theorem}
\label{thm:GK2}
Let $A$ be a $\mb{Z}$-graded simple finitely generated $k$-algebra which is a domain, where $k$ is
algebraically closed, $\GK A = 2$, and $A_i \neq 0$ for all $i \in \mb{Z}$.
\begin{enumerate}
\item $A$ is an Ore domain and $Q_{\rm gr}(A) \cong K[t, t^{-1}; \sigma]$, where $K = k(u)$
is a rational function field.  Moreover, $A_0 = T$ where either 
either (A) $T = k[u]$, $\sigma(u) = u+1$ and $\cha k = 0$; or (B) $T = k[u, u^{-1}]$, and $\sigma(u) =p u$ for some non-root of unity $p \in k^*$. 
 \item $A \cong \bigcap_{i = 1}^m B(G_i, H_i, T)$ for some points $p_i$ on distinct $\sigma$-orbits, pleasantly alternating cycles $G_i$ supported on the orbit of $p_i$, and ideals $H_i$.  
The algebra $A$ is graded Morita equivalent to a generalized Weyl algebra $B = B(Z, H, T) \cong T(\sigma, f)$, for 
$Z = \{p_1, \dots, p_m \}$ and some ideal $H = fT$ with $T/H$ supported along $Z$.
\end{enumerate}
\end{theorem} 
\begin{proof} 
(1) $A$ is Ore since it is a domain of finite GK-dimension.  Its graded quotient ring 
has the form $Q = Q_{\rm gr}(A) \cong K[t, t^{-1}; \sigma]$ for some finitely
generated field extension $K/k$ with $\trdeg K/k = 1$, by
Lemma~\ref{D-fg-lem}.
If there is $v \in K$ such that $\sigma(v) = v$ with $v \not \in k$,
then since $k$ is algebraically closed, $v$ is transcendental over
$k$, so $K/k(v)$ is a finite extension.  This forces $\sigma: K \to
K$ to have finite order, and then $Q = K[t, t^{-1}; \sigma]$ is a PI
ring, and so $A$ is as well.   This clearly contradicts the hypothesis that $A$ is simple.
Thus $\{x \in K | \sigma(x) = x \} =k$; that is, $Q$ is centerless.  
Then by Theorem~\ref{thm:inT}, $K = k(u)$ and there are $T$ and $\sigma$ falling into one of the two listed cases, with $A_0 = T$.

(2)  Note that Hypothesis~\ref{hyp:main} holds for $A$.   Let $X = \spec T$.  Since $A_0 = T$ is a PID, a 
$\Pic(X)$-twist does not change $A$ up to isomorphism, by Lemma~\ref{lem:triv-change}.   
Thus by Theorem~\ref{thm:class}, $A \cong \bigcap_{i = 1}^m B(G^{(i)}, H^{(i)}, J^{(i)})$, for 
some $Z = \{ p_1, \dots, p_m \}$ with each $p_i$ on a distinct $\sigma$-orbit, $G^{(i)}$ a pleasantly 
alternating cycle on the orbit of $p_i$, and ideals $H^{(i)}, J^{(i)}$.  
For $0 \neq x \in K$, we can do an explicit $\Pic(X)$-twist by replacing the ring
$A = \bigoplus_{n \in \mb{Z}} V_n t^n$ by $\bigoplus_{n \in \mb{Z}} x_n V_n t^n$, where $x_0 = 1$, 
$x_n = x \sigma(x) \dots \sigma^{n-1}(x)$ for $n > 0$, and $x_{-n} = [\sigma^{-1}(x) \dots \sigma^{-n}(x)]^{-1}$ 
for $n > 0$.  (One can also just interpret this as a change of variable $t$ in the quotient ring, replacing $t$ by $x^{-1} t$.)
If $J^{(i)} = x_i T$ and $H^{(i)} = y_i T$, then applying this twist with $x = (x_i)^{-1}$ 
to $B(G^{(i)}, H^{(i)}, J^{(i)})$
shows that 
$B(G^{(i)}, H^{(i)}, J^{(i)}) \cong B(G^{(i)}, \wh{H}^{(i)}, T)$, where 
$\wh{H}^{(i)} = x_i y_i T$.   Then applying the twist with $x = \prod (x_i)^{-1}$ to $A$ 
gives 
\[
A \cong \bigcap_{i = 1}^m B(G^{(i)}, H^{(i)}, J^{(i)}) \cong  \bigcap_{i = 1}^m B(G^{(i)}, \wh{H}^{(i)},  T),
\] 
as required.  Finally, this ring is graded Morita equivalent to $B(Z, H, T)$ for some $H$, 
by the same argument as in Theorem~\ref{thm:class}.

It remains to verify that a ring of the form $B(Z, H, T)$ is isomorphic to a generalized Weyl algebra 
$T(\sigma, f)$.  Since $T$ is a PID we may write $H = gT$.  Then we claim that 
$B = B(Z, H, T) \cong T(\sigma, f)$ for $f = \sigma^{-1}(g)$.  
By the original definition of Bavula \cite{Bav}, $T(\sigma, f)$ is the algebra generated by the ring $T$ and new indeterminates $x, y$ satisfying the relations $yx = f, xy = \sigma(f), x\alpha = \sigma(\alpha)x$, and $y \alpha = \sigma^{-1}(\alpha) y$ for all $\alpha \in T$.
Mapping $T \langle x, y \rangle$ to $T[t, t^{-1}; \sigma]$ by sending $T$ to itself identically, $x \mapsto t$ and $y \mapsto ft^{-1}$, it is easy to see that these relations are satisfied, so there is a map of graded algebras 
$\phi: T(\sigma; f) \to T[t, t^{-1}; \sigma]$.  The image of this map is the algebra generated by $T$, $ft^{-1}$, and $t$, which is also 
generated by $fT t^{-1} \oplus T \oplus Tt = B_{-1} \oplus B \oplus B_1$, and hence is equal to $B$ since $B$ is generated in these degrees 
by Lemma~\ref{lem:B-props1}.   Finally, $\phi$ must be injective since $T(\sigma, f)$ is simple given that $f$ does not have more than one root on a given $\sigma$-orbit \cite[Corollary 3.2]{Bav}.  This proves that $B(Z, H, T) \cong T(\sigma; f)$ as claimed.
\end{proof}

\section{Lonely subsets}

In this final section, we take a stab at getting a better understanding of which closed subsets $Z$ of an affine variety $X = \spec R$ with wild automorphism $\sigma$ can be $\sigma$-lonely, and 
thus knowing more about what examples actually occur in our main classification result.
We might as well stick to the case of the only examples of finite-type affine varieties with wild automorphisms we know, namely the algebraic groups with translation automorphisms described in Proposition~\ref{rat-prop}.   In the next result, we study $\sigma$-lonely subsets of codimension 1 for these examples.  
\begin{theorem}
\label{thm:lonely}
Let $T$, $X = \spec T$, and $\sigma: T \to T$, be of type (A) or (B) as in Proposition~\ref{rat-prop}.  Suppose that $0 \neq f \in T$ is not a unit and that $Z = V(f)$ is the vanishing set of $f$ in $X$.  
\begin{enumerate}
\item In type (A), where $T = k[x_1, x_2^{\pm 1}, \dots, x_d^{\pm 1}]$ and $\sigma(x_1) = x_1 + 1$, 
$\sigma(x_i) = p_i x_i$ for $i \geq 2$, then $Z$ is $\sigma$-lonely if and only if (after replacing $f$ by a unit multiple) 
we have either (i) $f \in k[x_1]$ and no two roots of $f$ differ by an integer, or else (ii) there is $z = x_2^{i_2} \dots x_d^{i_d}$ for some 
$i_j \in \mb{Z}$ such that $f \in k[z, z^{-1}]$ and no two roots of $f$ (as a polynomial in $z$) have ratio which is a power $\rho^i$ of $\rho = p_2^{i_2} \dots p_d^{i_d}$ for some $i \in \mb{Z}$.

\item In type (B), if we label the variables starting with $x_2$ so that $T = k[x_2^{\pm 1}, \dots, x_d^{\pm 1}]$,  then we can think of this ring as a subring of the one considered in part (1) with $\sigma$ being the restriction.  
then $Z = V(f)$ is $\sigma$-lonely if and only if condition (ii) holds as above.
\end{enumerate}
\end{theorem}
\begin{proof}
Part (2) is an easy consequence of the proof of Part (1), so we concentrate on type (A).  
Note that the units group of $T$ is $\{ \beta x_2^{i_2} \dots x_d^{i_d} | 0 \neq \beta \in k, i_j \in \mb{Z} \}$.
By multiplying by a unit, we may assume that $f$ has the form $f = \sum_{n = 0}^{m} h_n(x_2, \dots, x_d) x_1^n$ for some $ m\ge 0$, 
where $h_m \neq 0$ and $h_m \in k[x_2, \dots, x_d]$ 
and has constant term $1$.

Assume that $Z$ is  $\sigma$-lonely; equivalently, that $Tf+T\sigma^n(f)=T$ for every $n \neq 0$.
Let $C = k[t_1, t_2^{\pm 1}, \ldots, t_d^{\pm 1}]$ and let 
$B = T \otimes_k C = k[ x_1, x_2^{\pm 1},\ldots ,x_d^{\pm 1}, t_1, t_2^{\pm 1}, \ldots ,t_d^{\pm 1}]$.
Let $g=f(x_1 + t_1, t_2 x_2,\ldots ,t_dx_d)\in B$.  Consider the ring $\overline{B} = B/(f, g)$.

We claim that $(fB+gB)\cap C \neq 0$.  
Suppose this does not hold.  Then $C$ embeds in $\overline{B}$ and we identify it with its image.
By generic freeness \cite[Theorem 14.4]{Ei}, there is a nonzero polynomial $q\in C$ 
such that $\overline{B}_q$ is a free $C_q$-module.
The orbit of the point $(0, 1, 1, \dots, 1) \in X$ under the automorphism $\sigma$ 
is the set $\{ (n, p_2^n, p_3^n, \dots, p_d^n) | n \in \mb{Z} \}$, and we know this is dense in $X$ 
since $\sigma$ is wild by Lemma~\ref{T-lem}(1).  Thus 
there is an infinite set $S$ of natural numbers $n$ such that $q(n, p_2^n,\ldots ,p_d^n) \neq 0$.  Then for $n\in S$,
the elements $t_1-n, t_2 - p_2^n, \ldots ,t_d-p_d^n$ generate a proper ideal $I$ of $C_q$.  Since
$\overline{B}_q$ is a free $C_q$-module, we see that $I$ lifts to a proper ideal of
$\overline{B}_q$.  In other words, 
$$fB+gB+(t_1-n)B+(t_2 - p_2^n)B + \cdots +(t_d-p_d^n)B \neq B$$ for $n\in S$.
Notice, however, that
$$B/(f,g, t_1 - n, t_2-p_2^n,\ldots ,t_d-p_d^n) \cong T/(f, \sigma^n(f)),$$ which contradicts the fact that $f$ and $\sigma^n(f)$ generate the unit ideal in $T$.  This proves the claim.

Now there exist $a,b\in B$ and nonzero $c(t_1,\ldots ,t_d)\in C$ such that 
$$af+bg=c(t_1,\ldots ,t_d).$$
For $(\alpha_1,\ldots ,\alpha_d) \in Z = V(f)$ we have
$$g(\alpha_1, \dots, \alpha_d, t_1, \dots, t_d)  = f(t_1 + \alpha_1, \alpha_2 t_2,\ldots ,\alpha_d t_d) \big| c(t_1,\ldots ,t_d)$$ 
in $C$.  
Up to multiplication by units, $c(t_1,\ldots ,t_d)$ has only a finite set of divisors in $C$, say $\{c_1,\ldots ,c_m\}$.
By the choice of $f = \sum_{n = 0}^{m} h_n x_1^n$, where $h_m \in k[x_2, \dots, x_n]$ has constant term $1$, 
we see that 
$f(t_1 + \alpha_1, \alpha_2 t_2, \ldots, \alpha_d t_d)$ is of the form 
$\sum_{n = 0}^m h'_n(t_2, \dots, t_d) t_1^n$ where $h'_m \in k[t_2, \dots, t_n]$ still has constant term $1$.
Note that for each $c_i$, there is a finite set of units $u$ in the units group 
$\{ \beta t_2^{i_2} \dots t_d^{i_d} \}$ of $C$ such that $u c_i$ has constant term $1$.
Thus we see that 
$$S = \{f(t_1 + \alpha_1, \alpha_2 t_2,\ldots ,\alpha_d t_d)~|~(\alpha_1,\ldots ,\alpha_d)\in Z = V(f)\}$$ 
is a finite set.

Suppose that $f$ involves both $x_1$ and some $x_i$ with $i \geq 1$ (so in particular, $m \geq 1$ and $d \geq 2$).  
We will show this leads to a contradiction.
By reordering the $x_i$ with $i \geq 2$ we can assume that if $k$ is the largest integer such that $h_k$ is not a scalar, 
then $x_2$ occurs in $h_k$.   Consider the projection morphism  
$X \to \spec k[x_3^{\pm 1}, \dots, x_d^{\pm 1}]$, and let $\phi: Z(f) \to \spec k[x_3^{\pm 1}, \dots, x_d^{\pm 1}]$
be its restriction.  Since $Z(f)$ is $(d-1)$-dimensional, we can find a fiber of $\phi$, say  $\phi^{-1}(\alpha_3, \dots, \alpha_d)$, which is of dimension $\geq 1$ \cite[Exercise II.3.22]{Ha}.
Since $k$ is algebraically closed and in particular infinite, there is an infinite set $P$ of pairs $(\alpha, \beta)$ such that 
$(\alpha, \beta, \alpha_3, \dots, \alpha_d) \in Z(f)$ and 
the elements $f(t_1 + \alpha, \beta t_2, \alpha_3 t_3, \dots, \alpha_d t_d)$ are all equal.
Given any two distinct $(\alpha, \beta), (\gamma, \delta) \in P$, looking at the coefficient of $t_1^m$ implies that 
$h_m(\beta t_2, \alpha_3 t_3, \dots, \alpha_d t_d) = h_m(\delta t_2, \alpha_3 t_3, \dots, \alpha_d t_d)$; call this element $H$, and 
note that the constant term of $H$ is equal to $1$. 
Then looking at the coefficient of $t_1^{m-1}$ gives 
\[
H m \alpha + h_{m-1}(\beta t_2, \alpha_3 t_3, \dots, \alpha_d t_d) = H m \gamma + h_{m-1}(\delta t_2, \alpha_3 t_3, \dots, \alpha_d t_d).
\]
Looking at the constant term on both sides now shows since $k$ has characteristic $0$ that $\alpha = \gamma$.  
Next, let $k$ be the largest integer such that $h_k$ is non-scalar, and recall that $x_2$ occurs in $h_k$ by assumption.
Then looking at the coefficient of $t_1^k$ and using that $\alpha = \gamma$ 
we see that $M + h_k(\beta t_2, \alpha_3 t_3, \dots, \alpha_d t_d) = M + h_k(\delta t_2, \alpha_3 t_3, \dots, \alpha_d t_d)$ 
for some constant $M$.  Now looking at the term of highest degree in $t_2$, say degree $p$, we conclude that $\beta^p = \delta^p$.  In particular, $\beta = \zeta \delta$ for some $p$th root of unity $\zeta$.  
We conclude that $P \subseteq \{ (\alpha, \zeta \beta) | \zeta^p = 1 \}$ for some fixed $\alpha, \beta, p$, and so 
$P$ is finite, a contradiction.

Thus we can assume now that either $f \in k[x_1]$ or $f \in k[x_2, \dots, x_d]$.  The first case is easy to dispatch.
If $f \in k[x_1]$ then $f T + \sigma^n(f) T$ is the unit ideal if and only if $f k[x_1] + \sigma^n(f) k[x_1]$ is the unit 
ideal in $k[x_1]$.  This will occur for all $n \neq 0$ if and only if $f$ does not have two distinct roots on the same $\sigma$-orbit, 
in other words no two roots of $f$ differ by a nonzero integer.  

The other case is $f \in k[x_2^{\pm 1}, \dots, x_d^{\pm 1}]$, and this is handled by a similar argument as the one we already used.  
The argument in this paragraph also shows how to prove part (2) of the theorem; since case (B) allows arbitrary characteristic, we note that 
no assumption on the characteristic is necessary for this part of the argument.   We claim that there is $z = x_2^{i_2} \dots x_d^{i_d}$ such that $f \in k[z, z^{-1}]$.    We may assume that $d \geq 3$, since otherwise the claim is trivial.
Write $f(t_2,\ldots ,t_d)=\sum \beta_{i_2,\ldots ,i_d} t_2^{i_1}\cdots t_d^{i_d}$, where by the 
earlier normalization, $\beta_{0,\ldots ,0}=1$. 
Suppose there exist $(i_2,\ldots ,i_d)$ and $(j_2,\ldots ,j_d)$ in $\mathbb{Z}^{d-1} 
\setminus \{(0,0,\ldots ,0)\}$ such that $\beta_{i_2, \ldots, i_d} \neq 0$ and $\beta_{j_2, \ldots, j_d} \neq 0$, 
and $(i_2,\ldots ,i_d)$ is not a rational scalar multiple of $(j_2,\ldots ,j_d)$.   Then by relabeling our variables if necessary, we may assume that $i_2j_3 \neq j_2 i_3$.   Consider the projection $\spec k[x_2^{\pm 1}, x_3^{\pm 1}, \dots, x_d^{\pm 1}] \to \spec k[x_4^{\pm 1}, \dots, x_d^{\pm 1}]$
and restrict this to $Z(f)$ to obtain $\phi: Z(f) \to  \spec k[x_4^{\pm 1}, \dots, x_d^{\pm 1}]$.  Similarly as above, 
we can pick $(\alpha_4,\ldots ,\alpha_d) \in (k\setminus \{0\})^{d-3} $ such that $\phi^{-1}(\alpha_4,\ldots ,\alpha_d)$ is at least $1$-dimensional, 
and so  there is an infinite set $P$ of ordered pairs $(\alpha,\beta)\in (k\setminus\{0\})^2$ such that 
$f (\alpha t_2,\beta t_3 ,\alpha_4 t_4,\ldots ,\alpha_d t_d)$ is the same polynomial for every $(\alpha,\beta) \in P$.  
Now if $(\alpha,\beta)$ and $(\gamma,\delta)$ are in $P$, then 
we have $\alpha^{i_2}\beta^{i_3}=\gamma^{i_2}\delta^{i_3}$ and
$\alpha^{j_2}\beta^{j_3}=\gamma^{j_2}\delta^{j_3}$.   
Then $(\alpha/\gamma)^{i_2} = (\delta/\beta)^{i_3}$ and $(\alpha/\gamma)^{j_2} = (\delta/\beta)^{j_3}$, 
so there are roots of unity $\omega$ and $\omega'$ such that $\gamma=\omega\alpha$ and $\delta=\omega'\beta$ with $\omega^n=(\omega')^n=1$, where $n=i_2j_3-i_3j_2 \neq 0$.  Then we see that
$$ P \subseteq \{(\alpha \omega, \beta \omega')~:~\omega^n=(\omega')^n=1\},$$ 
for some fixed pair $(\alpha, \beta)$, contradicting the fact that $P$ is infinite.  
We conclude that all $(j_2, \dots, j_d)$ such that $\beta_{(j_2, \dots, j_d)} \neq 0$ lie on a line through the origin, 
and so there is some nonzero $d$-tuple $(j_2,\ldots ,j_d)\in \mathbb{Z}^d$ such that 
setting $z = x_2^{j_2} \dots x_d^{j_d}$, then $f \in k[z, z^{-1}]$.  We may assume further that $\gcd(j_2, \dots, j_d) = 1$.

Finally, since $\sigma(z) = \rho z$ where $\rho = p_2^{j_2} \dots p_d^{j_d}$, we have reduced again to a one-variable 
case.  Clearly $f$ will be $\sigma$-lonely if and only if $f = f(z)$ does not have two distinct roots on any 
$\sigma$-orbit, that is, no two distinct roots of $f$ have ratio equal to $\rho^i$ for some $i \in \mb{Z}$.
\end{proof}

\begin{example}
Suppose we are in type (A) or (B) of Proposition~\ref{rat-prop}.  If $T$ has dimension $\leq 2$, then we can describe all $\sigma$-lonely 
subsets $Z$ of $X = \spec T$.  In case $\dim T = 1$ this is rather trivial: $Z$ is lonely if and only if no two of its distinct points lie on the same $\sigma$-orbit.  In case $\dim T = 2$, then as $T$ is a UFD, $Z$ is a disjoint union  
$Z = V(f) \cup W$ for some nonzero nonunit $f \in T$ and some finite set of points $W$.  It is easy to see that $Z$ will be lonely if and only if $V(f)$ 
satisfies the conclusion of Theorem~\ref{thm:lonely}, and each point of $W$ lies on a distinct $\sigma$-orbit disjoint from the $\sigma$-iterates of $V(f)$.
\end{example}

We close with the following open question. 
\begin{question}
Is there a simple classification of the $\sigma$-lonely subsets of $X = \spec T$ for $(T, \sigma)$ as in Proposition~\ref{rat-prop} with 
$T$ of arbitrary dimension?  
\end{question}

\providecommand{\bysame}{\leavevmode\hbox
to3em{\hrulefill}\thinspace}
\providecommand{\MR}{\relax\ifhmode\unskip\space\fi MR }
\providecommand{\MRhref}[2]{%
  \href{http://www.ams.org/mathscinet-getitem?mr=#1}{#2}
} \providecommand{\href}[2]{#2}

\end{document}